\def\SIAM{0}
\def\TRUE{1}
\def\FALSE{0}

\ifx\SIAM\FALSE
\documentclass{amsart}
\else
\documentclass[review]{siamart190516}
\fi

\usepackage{amsmath,amssymb,amsfonts,mathrsfs}
\usepackage{mathabx} 

\usepackage{microtype}
\DisableLigatures[f]{encoding = *, family = * }
\usepackage{subcaption}
\usepackage{marginnote} 

\ifx\SIAM\FALSE
\usepackage{textcomp}
\numberwithin{equation}{section}
\fi

\usepackage{tabularx}
\usepackage{multirow}
\usepackage{enumitem}
\usepackage{graphicx}
\usepackage{tikz}

\usepackage{pgfplots}
\usepackage{pgfplotstable}
\usetikzlibrary{through,calc,arrows}
\usepackage{float}
\usepackage{bm}

\newcommand*{\tran}{\mathsf{T}}

\usepackage{xcolor}

\usepackage{verbatim,booktabs}

\ifx\SIAM\FALSE
\usepackage{lineno}
\fi
\allowdisplaybreaks

\ifx\SIAM\FALSE
\usepackage[pdftex,
	    pdftitle={Uniform convergence of an upwind discontinuous Galerkin method for solving scaled discrete-ordinate radiative transfer equations with isotropic scattering kernel},
            unicode=true,
            colorlinks,
            linkcolor=red,
            anchorcolor=blue,
            citecolor=green
            ]{hyperref}
\usepackage{cleveref}
\else
\hypersetup{pdftitle={Uniform convergence of an upwind discontinuous Galerkin method for solving scaled discrete-ordinate radiative transfer equations with isotropic scattering kernel}}
\fi

\ifx\SIAM\FALSE
\newtheorem{theorem}{Theorem}[section]
\newtheorem{lemma}[theorem]{Lemma}
\newtheorem{corollary}[theorem]{Corollary}

\theoremstyle{definition}
\theoremstyle{remark}\newtheorem{remark}[theorem]{Remark}
\theoremstyle{definition}\newtheorem{example}{Example}[section]
\theoremstyle{definition}\newtheorem{prob}{Problem}[section]
\newtheorem{assumption}{Assumption}
\else
\newsiamthm{assumption}{Assumption}
\newsiamremark{remark}{Remark}
\newsiamthm{prob}{Problem}
\newsiamremark{example}{Example}

\fi

\newcommand{\ud}{\,\mathrm{d}}
\newcommand{\bx}{\bm{x}}
\newcommand{\bomega}{\bm{\omega}}
\newcommand{\bOmega}{\bm{\Omega}}

\newcommand{\veps}{\varepsilon}

\ifx\SIAM\FALSE
\title[Uniform convergence of DG scheme for radiative transfer]{Uniform convergence of an upwind discontinuous Galerkin method for solving scaled discrete-ordinate radiative transfer equations with isotropic scattering kernel}
\author[Q.~Sheng]{Qiwei Sheng}
\address{Department of Mathematics, California State Univeristy, Bakersfield, CA 93311}
\email{qsheng@csub.edu}
\author[C.~Hauck]{Cory D.~Hauck}
\address{Computational and Applied Mathematics Group, Oak Ridge National Laboratory, Oak Ridge, TN 37831}
\email{hauckc@ornl.gov}
\thanks{This material was based, in part, upon work supported by the DOE Office of Advanced Scientific Computing Research and by the National Science Foundation under Grant No. 1217170.  ORNL is operated by UT-Battelle, LLC., for the U.S. Department of Energy under Contract DE-AC05-00OR22725.  The United States Government retains and the publisher, by accepting the article for publication, acknowledges that the United States Government retains a non-exclusive, paid-up, irrevocable, world-wide license to publish or reproduce the published form of this manuscript, or allow others to do so, for the United States Government purposes. The Department of Energy will provide public access to these results of federally sponsored research in accordance with the DOE Public Access Plan (\texttt{http://energy.gov/downloads/doe-public-access-plan}). }
\else
\title{Uniform convergence of an upwind discontinuous Galerkin method for solving scaled discrete-ordinate radiative transfer equations with isotropic scattering kernel
	\thanks{This material was based, in part, upon work supported by the DOE Office of Advanced Scientific Computing Research and by the National Science Foundation under Grant No. 1217170.  ORNL is operated by UT-Battelle, LLC., for the U.S. Department of Energy under Contract DE-AC05-00OR22725.  The United States Government retains and the publisher, by accepting the article for publication, acknowledges that the United States Government retains a non-exclusive, paid-up, irrevocable, world-wide license to publish or reproduce the published form of this manuscript, or allow others to do so, for the United States Government purposes. The Department of Energy will provide public access to these results of federally sponsored research in accordance with the DOE Public Access Plan (\texttt{http://energy.gov/downloads/doe-public-access-plan}). }
}
\author{Qiwei Sheng\thanks{Department of Mathematics, California State University, Bakersfield, Bakersfield, CA 93311 (\email{qsheng@csub.edu})}
        \and
        Cory D.~Hauck\thanks{Computational and Applied Mathematics Group, Oak Ridge National Laboratory, Oak Ridge, TN 37831, and Department of Mathematics, University of Tennessee, Knoxville, TN 37996
 (\email{hauckc@ornl.gov})}
        }
\headers{Uniform convergence of DG scheme for radiative transfer}{Q.~Sheng and C.~Hauck}
\fi

\ifx\SIAM\TRUE
\begin{document}
\maketitle


\else
\begin{document}
\fi

\begin{abstract}
We present an error analysis for the discontinuous Galerkin method applied to the  discrete-ordinate discretization of the steady-state radiative transfer equation. Under some mild assumptions, we show that the DG method converges uniformly with respect to a scaling parameter $\varepsilon$ which characterizes the strength of scattering in the system.  However, the rate is not optimal and can be polluted by the presence of boundary layers.  In one-dimensional slab geometries, we demonstrate optimal convergence when boundary layers are not present and analyze a simple strategy for balance interior and boundary layer errors.  Some numerical tests are also provided in this reduced setting. 
\end{abstract}

\ifx\SIAM\FALSE
\maketitle
\else


\fi

\section{Introduction}

The radiative transfer equation (RTE) is a kinetic equation that describes the scattering and absorption of radiation through a material medium.  It plays an important role in a wide range of applications such as astrophysics \cite{Pe2004}, atmosphere and ocean modeling \cite{CY2014,TS1999, ZTB}, heat transfer \cite{Modest}, and neutron transport and nuclear physics \cite{CZ1967, DM1978}.

Given a positive scalar parameter $\varepsilon \ll 1$, a highly diffusive medium is characterized by a large ($O(\veps^{-1})$) scattering cross-section, a small ($O(\veps)$) absorption cross-section, and a small  ($O(\veps)$) volumetric source.  For such media, the RTE can be approximated in the interior of the spatial domain, i.e., away from boundaries and discontinuous material interfaces, by a much simpler diffusion equation.   In particular, the RTE solution converges in the limit as $\veps \to 0$ to the solution of a diffusion equation that is independent of $\veps$, with boundary conditions that are determined by the solution of a half-space problem \cite{habetler1975uniform,LarsenKeller,bensoussan1979boundary}.

In many applications, material cross-sections can vary significantly in space and even in time so that both diffusive and non-diffusive regions coexist.  In such situations, kinetic models like the RTE are necessary for accurate simulations, but traditional discretization approaches for the RTE may not be efficient in diffusive regions.  Indeed, local truncation errors may scale like $h^p / \varepsilon$, where $h>0$ is the spatial mesh size and $p>0$ is an integer related to the formal  order of the method \cite{jin2010asymptotic,lowrie2002methods}.   Thus for fixed $h$, the accuracy of the method degrades dramatically as $\veps \to 0$.


Discontinuous Galerkin (DG) methods were first introduced for the purpose of simulating radiation transport equations \cite{reed1973triangular} and are now commonly used in combination with discrete-ordinate angular discretizations of the RTE \cite{LM2010}.  DG methods for the RTE without scattering were first analyzed in a rigorous fashion in \cite{LaSaint1974Finite}.  Later a complete space-angle convergence analysis for problems with scattering was conducted in \cite{HHE2010}.  While the analysis is \cite{reed1973triangular, HHE2010} is valid for fixed $\veps = O(1)$, it does not address the behavior of the method in multi-scale settings, when $\veps$ varies.  Indeed, a major benefit of DG methods is that they are accurate in the diffusion limit (when $\veps \to 0$), assuming a sufficiently rich approximation space.  Such behavior was first investigated in \cite{LM1989}, where it was shown that linear DG elements are sufficient to capture the diffusion limit for discrete-ordinate systems in one-dimensional slab geometries.  These results in \cite{LM1989} were later extended to multi-dimensional settings in \cite{Adams2001}, where it was shown that the DG approximation space must support global linear functions in order to capture the diffusion limit correctly.  This requirement translates to local $P_1$ approximations for triangular elements and local $Q_1$ approximations for rectangular elements.   The results in \cite{Adams2001} were later re-examined and formalized in a functional analysis framework in \cite{GK2010}, where the behavior of limiting discretization was explored in detail, including the presence of boundary layers.   Together, conventional error bounds and asymptotic results can be combined to construct uniform errors estimates \cite{jin2010asymptotic}.  However, such estimates for the RTE will suffer from a factor of two reduction in order.

In this paper, we investigate the convergence of DG discretizations of the discrete-ordinate RTE with variable $\veps$, and we derive error estimates that imply uniform convergence for the upwind DG scheme in the case  of isotropic (constant-in-angle) boundary conditions. A consequence of this result is that the method is necessarily accurate in the diffusion limit.  In general, the uniform errors obtained come at a price:  whereas the error obtained in \cite{HHE2010} for an order $k$ method with fixed $\veps$ is $O(h^{k+1/2})$, here we obtain a uniform $O(h^{k})$ bound across all  $\veps \in (0,1]$.  This bound improves upon the strategy in \cite{jin2010asymptotic} (which in general leads to $O(h^{k/2 + 1/4})$ bounds) whenever $k \geq1$.  However, a uniform $O(h^{k+1})$ estimate can be obtained in one-dimensional slab geometries.  When the incoming boundary condition is not isotropic, a boundary layer effect occurs.   In an approach similar to the one taken in \cite{M2013}, we balance the interior discretization error with the error introduced by the unresolved boundary layer.  By introducing an auxiliary problem, we can still obtain a uniform (although reduced) error estimate, at least in the slab geometry setting.  

The rest of this article is organized as follows. In \Cref{sec:settings}, the radiative transfer equation and its discrete-ordinate discretization are introduced, and the notations used in the remainder of the paper are set.  A priori estimates regarding the solutions of the discrete-ordinate equation are also presented.  In \Cref{sec:asym_error_analysis}, uniform convergence and error estimates of the discontinuous Galerkin method for the discrete-ordinate equation are established.  In \Cref{sec:1d_AP_error_analysis}, a simple strategy for handling anisotropic incoming boundary conditions in one-dimensional slab geometries is proposed and analyzed. Numerical experiments are presented in \Cref{sec:num_results}, and concluding remarks are given in \Cref{sec:conclusion}.

\section{Problem setting}\label{sec:settings}
Throughout the paper, the symbol $\lesssim$ abbreviates $a\le Cb$ for any two real quantities $a$ and $b$ with $C$ a positive nonessential constant independent of the finite element mesh size, which may take different values at different appearances. 
We also adopt the conventional notation $H^r(D)$ to indicate the Sobolev spaces on (possibly lower-dimensional) subdomain $D\subset X$ with the norm $\|\cdot\|_{r,D}$. Clearly, we have $H^0(D)=L^2(D)$ whose norm are denoted by $\|\cdot\|_D$. Define $H^1_0(X)=\{v\in H^1(X):v|_{\partial X}=0\}$. 
\subsection{The radiative transfer equation}
Let $X$ be a bounded domain in $\mathbb{R}^d$ with piecewise smooth boundary $\partial X$. For real physical problems, $d=3$.  However in certain geometries, reduced equations in one or two-dimensional can be derived \cite{LM1984}.
At each point $\bm{x}\in \partial X$ where the boundary is smooth, let $\bm{n}(\bm{x})$ be the unit normal to $\partial X$ that points outward from $X$.
Let  $\mathbb{S}$ be the the unit sphere in $\mathbb{R}^d$, and define the inflow and outflow boundary, respectively, by
\begin{equation}
  \Gamma_-=\{(\bm{x},{\bm{\omega}}) \in \partial X \times \mathbb{S}
  \colon \bm{\omega}\cdot \bm{n}(\bm{x}) <0\}, \quad \Gamma_+=\{(\bm{x},{\bm{\omega}}) \in \partial X \times \mathbb{S}
  \colon \bm{\omega}\cdot \bm{n}(\bm{x}) >0 \}.
\end{equation}

We consider the following scaled version of the steady-state RTE for the unknown $u = u(\bm{x},\bomega)$:
\begin{subequations}\label{eq:rte_scale_all}
\begin{alignat}{2}
 \bomega\cdot\nabla u+\frac{\sigma_{\mathrm{t}}}{\varepsilon}u
 &=\left(\frac{\sigma_{\mathrm{t}}}{\varepsilon}-\varepsilon\sigma_{\mathrm{a}}\right)\int_{\mathbb{S}}u\ud\breve{\bomega}+\varepsilon f, \quad & (\bm{x},{\bm{\omega}}) &\in X \times \mathbb{S} ,\label{eq:rte_scale}\\
 u&=\alpha, & (\bm{x},{\bm{\omega}}) &\in \Gamma_-,  \label{eq:rte_scale_bdry}
\end{alignat}
\end{subequations}
where $\sigma_{\mathrm{a}} = \sigma_{\mathrm{a}}(\bm{x})$ is the absorption cross section; $\sigma_{\mathrm{t}}=\sigma_{\mathrm{t}}(\bm{x})$ is the total cross section; $f = f(\bm{x})$ is a source that is assumed isotropic (i.e., independent of $\bomega$). Here $\alpha = \alpha(\bm{x},\bomega)$, and $\ud\breve{\bomega}$ is the normalized measure on $\mathbb{S}$, i.e., $\int_\mathbb{S} \ud\breve{\bomega}=1$.

We assume that
\begin{align}
  &\sigma_{\mathrm{t}},\sigma_{\mathrm{a}}\in L^\infty(X),\quad  \sigma_{\mathrm{t}}(\bx)-\varepsilon^2\sigma_{\mathrm{a}}(\bx) > 0  ~\text{a.e.}~ \bx \in X,\label{as:rte_scale_1}\\
  &\sigma_{\mathrm{a}} \ge \sigma_{\mathrm{a}}^{\mathrm{min}} \text{ in } X \text{   for a constant } \sigma_{\mathrm{a}}^{\mathrm{min}} >0,\label{as:rte_scale_1_0}\\
  &f(\bm{x})\in L^2(X) \text{ and } \alpha(\bm{x},\bomega)\in L^2(\Gamma_-).\label{as:rte_scale_1_1}
\end{align}
The quantity $\sigma_{\mathrm{t}}(\bx)-\varepsilon^2\sigma_{\mathrm{a}}(\bx)$ is the (non-dimensional) scattering cross-section. The condition \eqref{as:rte_scale_1_0}, while not strictly necessary, is often used in a priori estimates.

It is shown in \cite{agoshkov1998boundary} that under assumptions \cref{as:rte_scale_1,as:rte_scale_1_0,as:rte_scale_1_1}, the problem \eqref{eq:rte_scale_all} has a unique solution $u$ in the space
\begin{equation}
H^1_2(X\times\mathbb{S})=\{v \in L^2(X\times \mathbb{S})\colon
{\bomega}\cdot \nabla v \in L^2(X\times \mathbb{S})\}.
\end{equation}
The solution $u$ of \eqref{eq:rte_scale_all} depends on the parameter $\varepsilon>0$, which characterizes the relative strength of scattering in the system.  However, we suppress this dependence in our notation.

\subsection{Discrete-ordinate method}\label{subsec:DO}
The integral on the right-hand side of \eqref{eq:rte_scale} can be approximated by a numerical quadrature: Given $v \in C(\mathbb{S})$,
\begin{equation}\label{eq:quadrature}
 \int_{\mathbb{S}}v(\breve{\bomega})\ud\breve{\bomega}\approx \sum_{l=1}^L w_l v(\bomega_l).
\end{equation}
We assume that the ordinates $\{ \bomega_l \}$  and weights $\{ w_l \}$ are chosen such that
\begin{equation}
 \sum_{l=1}^L w_l=1,\quad \sum_{l=1}^L w_l \bomega_l=0, \text{ and } -\bomega_j\in\{\bomega_l\} \text{ whenever } \bomega_j\in\{\bomega_l\}.
\end{equation}

Discretization of \eqref{eq:rte_scale_all} in $\bomega$ using the quadrature formula \eqref{eq:quadrature} yields the discrete-ordinate equations:
\begin{subequations}\label{eq:rte_scale_DO_all}
\begin{align}\label{eq:rte_scale_DO}
 \bomega_l\cdot\nabla u^l +\frac{\sigma_{\mathrm{t}}}{\varepsilon}u^l &=\left(\frac{\sigma_{\mathrm{t}}}{\varepsilon}-\varepsilon\sigma_{\mathrm{a}}\right)\sum_{i=1}^L w_i u^i+ \varepsilon f^l,\\
 u^l&=\alpha^l,\quad \text{ on } \partial X^l_-, \quad1\le l\le L,\label{eq:rte_scale_DO_bdry}
\end{align}
\end{subequations}
where $ \alpha^l=\alpha(\bm{x},\bomega_l)$, $f^l \equiv f(\bm{x})$, and $u^l=u^l(\bm{x})$ is an approximation of $u(\bm{x},\bomega_l)$. Since $f$ is assumed isotropic, we often drop the superscript from $f^l$.
Here and below, we use the notation
\begin{equation}\label{eq:disc_bndry}
\partial X_-^l:=\{\bm{x}\in\partial X: \bomega_l\cdot\bm{n}(\bm{x})<0\},
\qquad 
\partial X_+^l:=\{\bm{x}\in\partial X: \bomega_l\cdot\bm{n}(\bm{x})>0\}, 
\end{equation}
and
\begin{equation}
	\partial X_- = \partial X_-^1\times \partial X_-^2\times \cdots \times \partial X_-^L, \quad \partial X_+ = \partial X_+^1\times \partial X_+^2\times \cdots \times \partial X_+^L.
\end{equation}

We now introduce some additional notation that will be used in the paper.  Let
\begin{alignat}{2}
 \bm{u}&=\begin{bmatrix} u^1 & u^2 & \cdots & u^L
        \end{bmatrix}^{\tran},
        \quad
  &\bm{\alpha}&=\begin{bmatrix}\alpha^1 & \alpha^2 & \cdots & \alpha^L
        \end{bmatrix}^{\tran}, \\
 \bOmega&=\begin{bmatrix}\bomega_1&\bomega_2&\cdots&\bomega_L
         \end{bmatrix}^{\tran},
         \quad
  &\bm{f}&=\begin{bmatrix}f &f&\cdots&f
  \end{bmatrix}^{\tran}_{1\times L},
\end{alignat}
and formally define
\begin{equation}\label{eq:dot_prod_1}
\bOmega\cdot\nabla\bm{u}=\begin{bmatrix}
\bomega_1\cdot\nabla u^1 & \bomega_2\cdot\nabla u^2 &\cdots & \bomega_L\cdot\nabla u^L
\end{bmatrix}^{\tran}.
\end{equation}
Then, for any vector $\bm{v}\in \mathbb{R}^L$, set
\begin{equation}\label{eq:rte_scale_avg_disc}
 \overline{v}=\sum_{l=1}^L w_l v^l
 \quad \text{and} \quad
 \overline{\bm{v}}=\begin{bmatrix}
 \overline{v} & \overline{v} & \cdots & \overline{v}
 \end{bmatrix}^{\tran}_{1\times L}.
\end{equation}

Finally, define the projection matrix $P \in \mathbb{R}^{L \times L}$ and the weight matrix $W \in \mathbb{R}^{L \times L}$ by
\begin{equation*}
 P=\begin{bmatrix}
    w_1 & w_2 & \cdots & w_{L-1} & w_L\\
    w_1 & w_2 & \cdots & w_{L-1} & w_L\\
    \vdots & \vdots & \ddots & \vdots & \vdots\\
    w_1 & w_2 & \cdots & w_{L-1} & w_L\\
   \end{bmatrix}_{L\times L} \quad
 W=\begin{bmatrix}
             w_1 & 0 & \cdots &  0\\
	     0 & w_2 & \cdots &  0\\
	     \vdots & \vdots & \ddots & \vdots\\
	     0 & 0 & \cdots & w_L\\
             \end{bmatrix}
\end{equation*}
respectively, and the discrete collision operator
\begin{equation*}
Q =\frac{\sigma_{\mathrm{t}}}{\varepsilon}I - \left(\frac{\sigma_{\mathrm{t}}}{\varepsilon}-\varepsilon\sigma_{\mathrm{a}}\right)P
= \frac{\sigma_{\mathrm{t}}}{\varepsilon}(I-P) + \varepsilon\sigma_{\mathrm{a}} P.
\end{equation*}
Given the assumptions $\sigma_{\mathrm{t}}$ and $\sigma_{\mathrm{a}}$ in \eqref{as:rte_scale_1_0},  it follows that for each fixed $\bm{x}\in X$, $Q$ is a strictly positive definite matrix with inverse
$Q^{-1}=\frac{1}{\varepsilon \sigma_{\mathrm{a}}}P +\frac{\varepsilon}{\sigma_{\mathrm{t}}}(I-P)$.
In addition, $Q$ is symmetric with respect to the weight $W$, i.e., $WQ=Q^{\tran} W$. 

Using  the notation above, \eqref{eq:rte_scale_DO_all} can be rewritten in the compact form
\begin{subequations}\label{eq:rte_scale_compact}
\begin{alignat}{3}
 \bOmega\cdot\nabla\bm{u} + Q\bm{u} &= \varepsilon\bm{f}, &~~~&\text{ in }~ X\label{eq:rte_scale_compact_a}\\
 \bm{u}&=\bm{\alpha}, &~~~&\text{ on }~ \begin{bmatrix}\partial X_-^1 & \partial X_-^2 &\cdots & \partial X_-^L\end{bmatrix}^{\tran}.\label{eq:rte_scale_compact_BC}
\end{alignat}
\end{subequations}
Here \eqref{eq:rte_scale_compact_BC} is understood as \eqref{eq:rte_scale_DO_bdry}. We have the following existence and uniqueness result of the discrete-ordinate equations \eqref{eq:rte_scale_compact} in $\left[H^1_2(X)\right]^L=\bigl\{\bm{v}\in[L^2(X)]^L;\;\bOmega\cdot\nabla\bm{v} \in [L^2(X)]^L\bigr\}$, where $\bOmega\cdot\nabla\bm{v}$ are the generalized directional derivatives along $\bOmega$. 
\begin{theorem}
	Under Assumptions 
	\ifx\SIAM\TRUE
		\cref{as:rte_scale_1,as:rte_scale_1_0,as:rte_scale_1_1}, 
	\else
		\eqref{as:rte_scale_1}--\eqref{as:rte_scale_1_1}, 
	\fi
	for any fixed $\varepsilon>0$, the problem \eqref{eq:rte_scale_compact} has a unique solution $\bm{u}\in \left[H^1_2(X)\right]^L$.
\end{theorem}

The space $\left[H^1_2(X)\right]^L$ has a well-defined trace, and for any Lipschitz domain $D\subseteq X$, the following integration by parts formula holds \cite[Corollary B.57]{EG2004}: $\forall u,v\in \bigl\{w\in L^2(X);\;\omega_l\cdot\nabla w \in L^2(X)\bigr\}$, 
\begin{equation}\label{eq:int_by_parts}
	\int_D \bomega_l \cdot\nabla u\, v\ud\bx = -\int_D u\, \bomega_l\cdot \nabla v + \int_{\partial D} \bomega_l\cdot\bm{n} uv\ud\bx, \; 1\le l\le L.
\end{equation}
\begin{lemma}\label{lem:cont}
	Let $E\subset X$ be a Lipschitz surface. Then we have, for $l=1,\cdots,L$,
	\begin{equation}
		\int_{E} \bomega_l\cdot\bm{n} \big(u|_{D_1}-u|_{D_2}\big)\ud\bx=0,
	\end{equation}
	which implies that if $\bomega_l\cdot\bm{n}\neq 0$, then $u$ is continuous on $E$ almost everywhere.
\end{lemma}
\begin{proof}
	Letting $v=1$ in \eqref{eq:int_by_parts} yields the divergence formula:
	\begin{equation}\label{eq:div}
	\int_{\partial D} \bomega_l\cdot\bm{n} u\ud\bx = \int_D \bomega_l \cdot\nabla u\ud\bx.
	\end{equation}
	Let $D\subseteq X$ be a Lipschitz domain such that $D=D_1\cup D_2$ with $D_1\cap D_2=\emptyset$ and $\overline{D}_1\cap\overline{D}_2=E$, i.e., $D_1$ and $D_2$ has a shared surface $E$. Without loss of generality, we assume $E$ is a plane and the unit normal $\bm{n}$ on $E$ points from $D_1$ to $D_2$. Then by the divergence formula \eqref{eq:div}, we have 
	\begin{align*}
	\int_{\partial D} \bomega_l\cdot\bm{n} u\ud\bx &= \int_{D} \bomega_l \cdot\nabla u\ud\bx = \int_{D_1} \bomega_l \cdot\nabla u \ud\bx + \int_{D_2} \bomega_l \cdot\nabla u \ud\bx \\
	&= \int_{\partial D_1} \bomega_l\cdot\bm{n} u \ud\bx + \int_{\partial D_2} \bomega_l\cdot\bm{n} u \ud\bx \\
	&= \int_{\partial D} \bomega_l\cdot\bm{n} u\ud\bx + \int_{E} \bomega_l\cdot\bm{n} \big(u|_{D_1}-u|_{D_2}\big)\ud\bx.
	\end{align*}
	Therefore
	\begin{equation}
	\int_{E} \bomega_l\cdot\bm{n} \big(u|_{D_1}-u|_{D_2}\big)\ud\bx=0.
	\end{equation}
\end{proof}

\subsection{Variational Formulation}

Given $\bm{u}, \bm{v}\in \left[L^2(X)\right]^L$,  define the inner product
\begin{equation}
(\bm{u},\bm{v})=\sum_{l=1}^Lw_l\int_X u^l\,v^l\ud x=\int_X \bm{u}^{\tran} W\bm{v}\ud\bm{x}
\end{equation}
as well as the norms 
\begin{equation}\label{eq:norm_def}
\|\bm{u}\|_{r,X}=\left(\sum_{l=1}^Lw_l\|u^l\|_{r,X}^2\right)^{1/2} \quad\text{ for } \bm{u}\in \left[H^r(X)\right]^L.
\end{equation}
When $r=0$, we omit the subscripts $0$ and $X$, i.e., $\|\bm{u}\|:=\|\bm{u}\|_{0,X}$.
Similarly, on the boundary, let 
\begin{equation}
	(\bm{u},\bm{v})_{\partial X_{\pm}}=\sum_{l=1}^Lw_l\int_{\partial X^l_{\pm}} |\bomega_l\cdot\bm{n}|\, u^l v^l\ud x
	\quad \text{and}\quad 
	\|\bm{u}\|_{r,\partial X_{\pm}}=\left(\sum_{l=1}^Lw_l\|u^l\|_{r,\partial X^l_{\pm}}^2\right)^{1/2},
\end{equation}
where $\partial X_-^l$ and $\partial X_+^l$ are defined in \eqref{eq:disc_bndry}. When $r=0$, we omit the subscripts $0$. Finally, set $\|\bm{u}\|_{\partial X}=\sqrt{ \|\bm{u}\|^2_{\partial X_-} +\|\bm{u}\|^2_{\partial X_+}}$.

Multiplying \eqref{eq:rte_scale_compact_a} by an arbitrary function $\bm{v}\in \left[H^1_2(X)\right]^L$, integrating over $X$, using integration-by-parts, and employing the boundary condition \eqref{eq:rte_scale_compact_BC}, we get a variational formulation of \eqref{eq:rte_scale_compact}:
 \begin{equation}\label{eq:rte_var_form}
  \mathfrak{a}(\bm{u},\bm{v})=\ell(\bm{v}),\quad \forall \bm{v}\in \left[H^1_2(X)\right]^L,
 \end{equation}
 where
 \begin{align} 
    \mathfrak{a}(\bm{u},\bm{v})&= - (\bm{u},\bOmega\cdot\nabla\bm{v}) + (\bm{u}, \bm{v})_{\partial X_+} + (Q\bm{u},\bm{v}),\label{eq:def_B}\\
  \ell(\bm{v})&= \varepsilon (\bm{f},\bm{v}) + (\bm{\alpha}, \bm{v})_{\partial X_{-}}.\label{eq:def_ell}
 \end{align}
 Since $\bm{u}\in \left[H^1_2(X)\right]^L$,
 after another integration-by-parts, the bilinear form $\mathfrak{a}(\bm{u},\bm{v})$ can be rewritten in the following form:
 \begin{equation}\label{eq:rte_bi_form_cont}
  \mathfrak{a}(\bm{u},\bm{v})=(\bOmega\cdot\nabla\bm{u},\bm{v}) + (\bm{u}, \bm{v})_{\partial X_-} + (Q\bm{u},\bm{v}).
 \end{equation}

We will denote by 
\begin{equation*}
(\bm{u},\bm{v})_Q=(Q\bm{u},\bm{v}) \quad\text{ and } \quad \|\bm{u}\|_Q=(\bm{u},\bm{u})_Q^{1/2}.
\end{equation*}
Since $Q$ is symmetric with respect to $W$ and strictly positive definite, $(\cdot,\cdot)_{Q}$ 
is an inner product. One verifies that, since $(\bm{u}-\overline{\bm{u}},\overline{\bm{u}})=0$,  
\begin{equation*}
\|\bm{u}\|^2_Q = \frac{1}{\varepsilon}\|\sigma_{\mathrm{t}}^{1/2}(\bm{u}-\overline{\bm{u}})\|^2 + \varepsilon \|\sigma_{\mathrm{a}}^{1/2}\overline{\bm{u}}\|^2 \label{eq:norm_equiv_1}
\end{equation*}
holds for all $\bm{u}\in \left[L^2(X)\right]^L$.
A direct calculation shows that
\begin{equation}\label{lem:B_stab} 
	\mathfrak{a}(\bm{u},\bm{u}) = \|\bm{u}\|^2_Q + \frac{1}{2}\|\bm{u}\|^2_{\partial X}.
\end{equation}

\subsection{A priori Estimates}

The purpose of this subsection is to derive some a priori estimates for the discrete-ordinate equation. 
We deduce an a priori estimate that is needed for the error analysis. We essentially follow the proof in \cite{GK2010}.
\begin{lemma}\label{lem:pri_1}
	Assume that $\bm{\alpha}(\bm{x},\bomega)=\bm{\alpha}_0(\bm{x})+\bm{\alpha}_1(\bm{x},\bomega)$ with $\bm{\alpha}_0(\bm{x})\in \left[H^{1/2}(\partial X)\right]^L$ and $\bm{\alpha}_1(\bm{x},\bomega)\in \left[L^2(\partial X_-)\right]^L$. Then there is a constant $c$, uniform with respect to $\varepsilon$, so that 
	\begin{equation}\label{lem:priori_est_1}
		\frac{1}{\varepsilon}\|\bm{u}-\overline{\bm{u}}\|^2 + \varepsilon\|\overline{\bm{u}}\|^2 + \|\bm{u}-\bm{\alpha}_0\|^2_{\partial X} \le c\left(\|\bm{f}\|^2 + \|\bm{\alpha}_0\|^2_{1/2, \partial X}\right)\varepsilon + \|\bm{\alpha}_1\|^2_{\partial X_-}.
	\end{equation}
\end{lemma}
\begin{remark}
	A reasonable choice of $\bm{\alpha}_0(\bm{x})$ is $\bm{\alpha}_0(\bm{x})=\overline{\bm{\alpha}}(\bm{x}) := \frac{1}{\pi} \int_{\bomega\cdot\bm{n}<0} |\bomega\cdot\bm{n}|\,\bm{\alpha}(\bm{x},\bomega)\ud \bomega$. However, $\bm{\alpha}_0(\bm{x})$ can be any function which is independent of $\bomega$ and we can take $\bm{\alpha}_0(\bm{x})=\bm{\alpha}(\bm{x},\bomega)$ when $\bm{\alpha}(\bm{x},\bomega)$ is isotropic. 
\end{remark}
\begin{proof}
	Let $\bm{m}\in \left[H^1(X)\right]^L$ solve
	\begin{equation}
		\left\{\begin{aligned}
			\int_X \nabla \bm{m}\cdot\nabla\bm{v}\ud \bm{x}&=0, &~&\forall \bm{v}\in \left[H^1_0(X)\right]^L, \\
			\bm{m}|_{\partial X}&=\bm{\alpha}_0.
		\end{aligned}\right.
	\end{equation}
	Then by standard elliptic theory (cf.~\cite[Proposition 2.10]{EG2004}),
	\begin{equation}
		\|\bm{m}\|_{1,X}\le c \|\bm{\alpha}_0\|_{1/2,\partial X}.
	\end{equation}
	Since $\bm{m}$ is isotropic, $(\bm{m},\bm{v})_Q=(\varepsilon\sigma_{\mathrm{a}}\bm{m},\bm{v})$, using the representation \eqref{eq:rte_bi_form_cont} and taking $\bm{v}=\bm{u}-\bm{m}$, we rewrite \eqref{eq:rte_var_form} as follows: 
	\begin{align*}
		\mathfrak{a}(\bm{u}-\bm{m},\bm{u}-\bm{m})
		&=\ell(\bm{u}-\bm{m})-\mathfrak{a}(\bm{m},\bm{u}-\bm{m}) \nonumber\\
		&=\varepsilon (\bm{f},\bm{u}-\bm{m}) + (\bm{\alpha}_1,\bm{u}-\bm{m})_{\partial X_-}
		- (\bOmega\cdot\nabla \bm{m} + \varepsilon\sigma_{\mathrm{a}} \bm{m}, \bm{u}-\bm{m}).
	\end{align*}
	By the fact that $\bm{m}$ is isotropic, the left-hand side of the above equation can be written as
	\begin{equation}
		\mathfrak{a}(\bm{u}-\bm{m},\bm{u}-\bm{m})=\frac{1}{\varepsilon}\|\sigma_{\mathrm{t}}^{1/2}(\bm{u}-\overline{\bm{u}})\|^2 + \varepsilon \|\sigma_{\mathrm{a}}^{1/2}(\overline{\bm{u}}-\bm{m})\|^2 + \frac{1}{2}\|\bm{u}-\bm{m}\|^2_{\partial X}.
	\end{equation}
	Therefore, we infer
	\begin{align}\label{eq:rte_pre_bd}
		\frac{1}{\varepsilon}\|\sigma_{\mathrm{t}}^{1/2}(\bm{u}&-\overline{\bm{u}})\|^2 + \varepsilon \|\sigma_{\mathrm{a}}^{1/2}(\overline{\bm{u}}-\bm{m})\|^2 + \frac{1}{2}\|\bm{u}-\bm{m}\|^2_{\partial X}\\
		&= \varepsilon (\bm{f},\bm{u}-\bm{m}) + (\bm{\alpha}_1,\bm{u}-\bm{m})_{\partial X_-}
		- (\bOmega\cdot\nabla \bm{m} + \varepsilon\sigma_{\mathrm{a}} \bm{m}, \bm{u}-\bm{m})\nonumber\\
		&:= \mathrm{I} + \mathrm{II} + \mathrm{III}. \nonumber
	\end{align}
	
	The rest of the proof consists of bounding each $\mathrm{I}$, $\mathrm{II}$ and $\mathrm{III}$ above. 
	For the first term, using the Cauchy-Schwarz inequality and noting that $\bm{f}$ is isotropic, we have
	\begin{align*}
		\mathrm{I} =\varepsilon (\bm{f},\bm{u}-\bm{m})
		&\le \varepsilon\|\sigma_{\mathrm{a}}^{-1/2}\bm{f}\| \|\sigma_{\mathrm{a}}^{1/2}(\overline{\bm{u}}-\bm{m})\| \nonumber\\
		&\le \varepsilon\left\|\sigma_{\mathrm{a}}^{-1}\right\|_{L^\infty(X)}\|\bm{f}\|^2 + \frac{\varepsilon}{4}\|\sigma_{\mathrm{a}}^{1/2}(\overline{\bm{u}}-\bm{m})\|^2.
	\end{align*}
	For the second term, since $\bm{u}-\bm{m}=\bm{\alpha}-\bm{\alpha}_0$, it follows that
	\begin{equation*}
		\mathrm{II}=(\bm{\alpha}_1,\bm{\alpha}_1)_{\partial X_-}
		=\|\bm{\alpha}_1\|_{\partial X_-}^2.
	\end{equation*}
	We handle the third term as follows:
	\begin{align*}
		\mathrm{III}&= - (\bOmega\cdot\nabla \bm{m} + \varepsilon\sigma_{\mathrm{a}} \bm{m}, \bm{u}-\overline{\bm{u}}+\overline{\bm{u}}-\bm{m}) \nonumber\\
		&= - (\bOmega\cdot\nabla \bm{m} + \varepsilon\sigma_{\mathrm{a}} \bm{m}, \bm{u}-\overline{\bm{u}}) - (\bOmega\cdot\nabla \bm{m} + \varepsilon\sigma_{\mathrm{a}} \bm{m}, \overline{\bm{u}}-\bm{m})\nonumber\\
		&= - (\bOmega\cdot\nabla \bm{m}, \bm{u}-\overline{\bm{u}}) - (\varepsilon\sigma_{\mathrm{a}} \bm{m}, \overline{\bm{u}}-\bm{m}),
	\end{align*}
	where, in the last step, we have used the fact that $(\varepsilon\sigma_{\mathrm{a}} \bm{m}, \bm{u}-\overline{\bm{u}})
	=\bm{0}$ and $(\bOmega\cdot\nabla \bm{m}, \overline{\bm{u}}-\bm{m})=\bm{0}$ due to the symmetry.
	Therefore,
	\begin{align*}
		|\mathrm{III}|
		&\le \|\sigma_{\mathrm{t}}^{-1/2}\nabla\bm{m}\| \|\sigma_{\mathrm{t}}^{1/2}(\bm{u}-\overline{\bm{u}})\| + \varepsilon\|\sigma_{\mathrm{a}}^{1/2}\bm{m}\| \|\sigma_{\mathrm{a}}^{1/2}(\overline{\bm{u}}-\bm{m})\| \nonumber\\
		&\le \frac{\varepsilon}{2}\left\|\sigma_{\mathrm{t}}^{-1}\right\|_{L^2(X)}\|\nabla\bm{m}\|^2 + \frac{1}{2\varepsilon}\|\sigma_{\mathrm{t}}^{1/2}(\bm{u}-\overline{\bm{u}})\|^2 \nonumber \\
		&\phantom{=}\quad + \varepsilon\left\|\sigma_{\mathrm{a}}\right\|_{L^2(X)}\|\bm{m}\|^2 + \frac{\varepsilon}{4}\|\sigma_{\mathrm{a}}^{1/2}(\overline{\bm{u}}-\bm{m})\|^2.
	\end{align*}
	
	By inserting all the bounds derived above into \eqref{eq:rte_pre_bd}, we finally infer that
	\begin{equation*}
		\frac{1}{\varepsilon}\|\sigma_{\mathrm{t}}^{1/2}(\bm{u}-\overline{\bm{u}})\|^2 + \varepsilon\|\sigma_{\mathrm{a}}^{1/2}(\overline{\bm{u}}-\bm{m})\|^2 + \|\bm{u}-\bm{m}\|^2_{\partial X}
		\le c(\|\bm{f}\|^2 + \|\bm{\alpha}_0\|^2_{1/2,\partial X_-})\varepsilon + 2\|\bm{\alpha}_1\|^2_{\partial X_-},
	\end{equation*}
	from which \eqref{lem:priori_est_1} follows easily since $\|\overline{\bm{u}}\|^2=\|\overline{\bm{u}}-\bm{m}+\bm{m}\|^2\le\|\overline{\bm{u}}-\bm{m}\|^2+\|\bm{m}\|^2$ and $\|\bm{m}\|_{\partial X_-}=\|\bm{m}\|_{\partial X_+}$ due to the isotropy of $\bm{m}$.
\end{proof}

\begin{remark}
	In \cite{GK2010}, the terms I and III are bounded in similar fashion.  However, because the analysis in \cite{GK2010} is discrete, the authors construct a bound for a discrete analog of II that is $O(\varepsilon/h)$, using an inverse inequality.
\end{remark}

\begin{remark}
	In case $f$ is not isotropic, define $\tilde{\bm{f}}=\bm{f}-\overline{\bm{f}}$. Then $\mathrm{I} =\varepsilon (\bm{f},\bm{u}-\bm{m}) = \varepsilon (\overline{\bm{f}},\bm{u}-\bm{m}) + \varepsilon (\bm{f}-\overline{\bm{f}},\bm{u}-\bm{m}) = \varepsilon (\overline{\bm{f}},\overline{\bm{u}}-\bm{m}) + \varepsilon (\bm{f}-\overline{\bm{f}},\bm{u}-\overline{\bm{u}})$. With an analogous argument, we have
	\begin{multline}\label{eq:priori_est_1_aniso}
		\frac{1}{\varepsilon}\|\bm{u}-\overline{\bm{u}}\|^2 + \varepsilon\|\overline{\bm{u}}\|^2 + \|\bm{u}-\bm{\alpha}_0\|^2_{\partial X} \\
		\le c\left(\|\overline{\bm{f}}\|^2 + \|\bm{\alpha}_0\|^2_{1/2, \partial X}\right)\varepsilon + \|\bm{\alpha}_1\|^2_{\partial X_-} + \varepsilon^3\|\bm{f}-\overline{\bm{f}}\|^2.
	\end{multline}
\end{remark}


Let us denote
\begin{equation}\label{def:delta}
	\delta:=\|\bm{\alpha}_1(\bm{x},\bomega)\|_{\partial X_-}.
\end{equation}
We have $\|\bm{u}-\bm{\alpha}_0\|_{\partial X_-}=\|\bm{\alpha}_1\|_{\partial X_-}=\delta\le \sqrt{C\varepsilon + \delta^2}$. 
\begin{corollary}\label{cor:est} 
	Let $\delta$ be given in \eqref{def:delta}. Then\footnote{The estimate for $\|\overline{\bm{u}}\|$ is obviously not optimal. In fact, when $\delta> 0$, as $\varepsilon\rightarrow 0$, $\|\overline{\bm{u}}\|\rightarrow \infty$, which is against our intuition. Better estimates for $\|\overline{\bm{u}}\|$ is possible. 
		But for our propose, the stated estimate is sufficient.}
	\begin{alignat*}{2}
		\|\bm{u}-\overline{\bm{u}}\| & \le \sqrt{C\varepsilon^2 + \delta^2\varepsilon},\quad &\|\overline{\bm{u}}\| &\le \sqrt{C +\frac{\delta^2}{\varepsilon}}, \\
		 \|\bm{u}-\bm{\alpha}_0\|_{\partial X_+} &\le \sqrt{C\varepsilon + \delta^2},\quad &\|\bm{u}-\bm{\alpha}_0\|_{\partial X} &\le \sqrt{C\varepsilon + \delta^2}.
	\end{alignat*}
	If we further assume that $\bm{\alpha}$ is isotropic, then
	\begin{equation}
		\|\bm{u}-\overline{\bm{u}}\|\le \sqrt{C}\varepsilon,\quad \|\overline{\bm{u}}\|\le \sqrt{C}, \quad\text{and\quad} \|\bm{u}-\bm{\alpha}_0\|_{\partial X}=\|\bm{u}-\bm{\alpha}_0\|_{\partial X_+}\le  \sqrt{C\varepsilon}.
	\end{equation}
\end{corollary}

\subsection{Discontinuous {G}alerkin method for spatial discretization}\label{sec:DG}
Let $\mathcal{T}_h$ be a regular family of triangulations/rectangulations of $X$. The meshes are assumed to be affine to avoid unnecessary technicalities, i.e., $X$ is assumed to be a polyhedron.
Define a finite element space by
\begin{equation}\label{eq:poly_space}
	V_h^k=\begin{cases}
	\{v\in L^2(\mathbb{S}); \;v|_K\in P_k(K)\;\forall K\in \mathcal{T}_h\}, &\text{ if } \mathcal{T}_h \text{ is triangular}, \\
	\{v\in L^2(\mathbb{S}); \;v|_K\in Q_k(K)\;\forall K\in \mathcal{T}_h\}, &\text{ if } \mathcal{T}_h \text{ is rectangular},
	\end{cases}
\end{equation}
where $k$ is a non-negative integer, $P_k(K)$ denotes the set of all polynomials on $K$ of a total degree no more than $k$, and
$Q_k(K) = \{\sum_j c_j p_j(x)q_j(y)r_j(z): p_j,q_j,r_j$ polynomials of degree $\le k\}$. Define $h_K:= \textrm{diam}(K)$ and set $h:=\max_{K\in \mathcal{T}_h}h_K$. For $\bx\in\partial K$, let $\bm{n}_K(\bx)$ be the outward unit normal to $K$. Let $\mathcal{E}^\mathrm{i}_h$ be the set of all interior faces of $\mathcal{T}$ and assign a unit normal direction $\bm{n}_e$ to each $e\in \mathcal{E}^\mathrm{i}_h$.
Define $\bm{V}_h=\left[V_h^k\right]^L$, and write a generic element in $\bm{V}_h$ as $\bm{v}_h=\{v^l_h\}^L_{l=1}$.

We introduce the following notation. For $\bm{u},\bm{v}\in \left[L^2(K)\right]^L$, $\left[L^2(e)\right]^L$, or $\left[L^2(\mathcal{E}^\mathrm{i}_h)\right]^L$, respectively,
\begin{eqnarray*}
 & (\bm{u},\bm{v})_K= \sum_{l=1}^L w_l\int_K u^l v^l\ud x = \int_K \bm{u}^{\tran} W \bm{v}\ud\bm{x},\quad \|\bm{u}\|_K=(\bm{u},\bm{u})_K^{1/2},\\
 & (\bm{u},\bm{v})_e= \sum_{l=1}^L w_l\int_e u^l v^l\ud x = \int_e \bm{u}^{\tran} W \bm{v}\ud\bm{x},\quad \|\bm{u}\|_e=(\bm{u},\bm{u})_e^{1/2},\\
 & (\bm{u},\bm{v})_{\mathcal{E}^\mathrm{i}_h} 
 =  \sum_{e\in \mathcal{E}^\mathrm{i}_h}\int_e (|\bOmega\cdot\bm{n}_e|\, \bm{u})^{\tran} W \bm{v}\ud\bm{x},\quad \|\bm{u}\|_{\mathcal{E}^\mathrm{i}_h}=(\bm{u},\bm{u})^{1/2}_{\mathcal{E}^\mathrm{i}_h}.
\end{eqnarray*}
Here, $|\bOmega\cdot\bm{n}_e|\, \bm{u} := \begin{bmatrix} |\bomega_1\cdot\bm{n}_e| u^1 & \cdots & |\bomega_L\cdot\bm{n}_e| u^L \end{bmatrix}^{\tran}$ (cf. \eqref{eq:dot_prod_1}).
For an interior face $e$, let $K^+$ and $K^-$ be two adjacent tetrahedrons sharing $e$ and the unit normal $\bm{n}_e$ points from $K^+$ to $K^-$. For a scalar-valued function $v$, we write $v^+=v|_{K^+}$ and $v^-=v|_{K^-}$. Then define the jump of $v$ on $e$ by $[\![ v ]\!]=v^+-v^-$ and the average of $v$ on $e$ by $\{\!\!\{ v \}\!\!\}=(v^+ + v^-)/2$.

The discontinuous {G}alerkin approximation $\bm{u}_h\in \bm{V}_h$ is locally defined by requiring, for each $K\in \mathcal{T}_h$, 
\begin{multline}\label{eq:dg_local}
(\bOmega \cdot \bm{n}_K \hat{\bm{u}}_h, \bm{v}_h)_{\partial K}
-(\bm{u}_h, \bOmega \cdot \nabla  \bm{v}_h )_K 
+ (Q \bm{u}_h, \bm{v}_h)_K \\
= \veps (\bm{f}, \bm{v}_h)_K, \; \forall\bm{v}_h\in \left[P_k(K)\right]^L \text{ or } \left[Q_k(K)\right]^L.
\end{multline}
Here, the upwind trace $\hat{u}^l_h$ is defined by the formula
\begin{equation*}
 \hat{u}^l_h(\bx)=\lim_{\tau\to 0^+} u^l_h(\bm{x}-\tau\bomega_l),\quad \bx\in \mathcal{E}^\mathrm{i}_h.
\end{equation*}
While other definitions of the numerical trace can be used, the upwind definition allows for increased computational efficiency via sweeping \cite{Larsen2010}.

To obtain a global formulation, we define the bilinear form $\mathfrak{a}_h: \bm{V}_h\times\bm{V}_h\to \mathbb{R}$ by 
\begin{multline}\label{eq:def_a}
\mathfrak{a}_h(\bm{u}_h,\bm{v}_h) = \sum_{e\in \mathcal{E}^\mathrm{i}_h} \int_{e}(\bOmega \cdot \bm{n}_e \hat{\bm{u}}_h)^{\tran} W [\![\bm{v_h}]\!] \ud \bm{x}
- \sum_K (\bm{u}_h, \bOmega \cdot \nabla  \bm{v}_h )_K \\
+ ( \bm{u}_h , \bm{v}_h)_{\partial X_+} + ( Q\bm{u}_h , \bm{v}_h).
\end{multline}
Let
$\bm{V}_{(h)}=\left[H^1_2(X)\right]^L + \bm{V}_{h}$. 
As discussed in Section~\ref{subsec:DO}, elements of $\left[H^1_2(X)\right]^L$ have a well-defined trace. Therefore, $\mathfrak{a}_h$ can be naturally extended to $\bm{V}_{(h)}\times \bm{V}_{h}$. 
After applying another integrating by parts applied to \eqref{eq:def_a}, the bilinear form $\mathfrak{a}_h$ can be rewritten in the following, which will be useful later:
\begin{multline}\label{eq:def_a_2}
\mathfrak{a}_h(\bm{u}_h,\bm{v}_h) := -\sum_{e\in \mathcal{E}^\mathrm{i}_h} \int_{e}\left(\bOmega \cdot \bm{n}_e [\![\bm{u_h}]\!]\right)^{\tran} W \check{\bm{v}}_h \ud \bm{x}
+ \sum_K  (\bOmega \cdot \nabla  \bm{u}_h, \bm{v}_h)_K \\
+ ( \bm{u}_h , \bm{v}_h)_{\partial X_-} + ( Q\bm{u}_h , \bm{v}_h),
\end{multline}
where the downwind trace $\check{u}^l_h$ is defined by
\begin{equation*}
\check{u}^l_h(\bm{x})=\lim\limits_{\tau\to 0^+} u^l_h(\bx+\tau\bomega_l),\quad \bx\in \mathcal{E}^\mathrm{i}_h.
\end{equation*}

The discrete-ordinate DG method can be cast as follows.
\begin{prob}
	Find $\bm{u}_h\in \bm{V}_h$ such that
	\begin{equation}\label{eq:rte_scale_DO_DG}
	 \mathfrak{a}_h(\bm{u}_h,\bm{v}_h) = \ell(\bm{v}_h), \quad \forall \bm{v}_h\in \bm{V}_{h},
	\end{equation}
	where $\mathfrak{a}_h$ and $\ell$ are given in \eqref{eq:def_a} and \eqref{eq:def_ell}, respectively. 
\end{prob}

\begin{lemma}[Stability]\label{lem:a_h_stab}
	For all $\bm{v}_h\in \bm{V}_{h}$,
	\begin{equation}\label{eq:rte_scale_stab}
	\mathfrak{a}_h(\bm{v}_h,\bm{v}_h)  = \|\bm{v}_h\|^2_Q + \frac{1}{2}\|\bm{v}_h\|^2_{\partial X} + \frac{1}{2}\big\|[\![\bm{v_h}]\!]\big\|^2_{\mathcal{E}^\mathrm{i}_h}.
	\end{equation}
\end{lemma}
\begin{proof}
	By using the relation
	\begin{equation*}
	\sum_K  (\bm{v}_h, \bOmega \cdot \nabla  \bm{v}_h )_K =  \sum_{e\in \mathcal{E}^\mathrm{i}_h} \big(\bOmega \cdot \bm{n}_e \{\!\!\{\bm{v}_h\}\!\!\}, [\![\bm{v_h}]\!]\big)_e
	+ \frac12 \|\bm{v}_h\|^2_{\partial X_+} - \frac12 \|\bm{v}_h\|^2_{\partial X_-}
	\end{equation*}
	and the fact that $\bOmega \cdot \bm{n}_e \hat{\bm{v}}_h = \bOmega \cdot \bm{n}_e \{\!\!\{\bm{v}_h\}\!\!\} + \frac12|\bOmega \cdot \bm{n}_e | \, [\![\bm{v}_h]\!]$, we find that
	\begin{multline*}
	\sum_{e\in \mathcal{E}^\mathrm{i}_h} \big(\bOmega \cdot \bm{n}_e \hat{\bm{v}}_h, [\![\bm{v_h}]\!]\big)_e - \sum_K  (\bm{v}_h, \bOmega \cdot \nabla  \bm{v}_h )_K \\ 
	= \frac12 \sum_{e\in \mathcal{E}^\mathrm{i}_h} \big(|\bOmega \cdot \bm{n}_e| \, [\![\bm{v_h}]\!], [\![\bm{v_h}]\!]\big)_e
	- \frac12 \|\bm{v}_h\|^2_{\partial X_+} + \frac12 \|\bm{v}_h\|^2_{\partial X_-},
	\end{multline*}
	from which \eqref{eq:rte_scale_stab} follows.
\end{proof} 
\section{Error Analysis}\label{sec:asym_error_analysis}
In this section, we establish bounds on the error $\bm{u}-\bm{u}_h$.

\subsection{Assumptions}
We  make the following assumptions
\begin{assumption}\label{assupt:reg_1}
	For some $r>1$, $\bm{u}\in \left[H^r(X)\right]^L$ 
	with
	\begin{equation}\label{eq:reg_assump}
	\|\bm{u}-\overline{\bm{u}}\|_{r,X}\le \sqrt{C\varepsilon^2 + \delta^2\varepsilon},\quad \|\overline{\bm{u}}\|_{r,X}\le \sqrt{C +\frac{\delta^2}{\varepsilon}}.
	\end{equation}
\end{assumption}
\begin{assumption}\label{assupt:reg_2}
	For some $r>0$,
	\begin{equation}
	\|\bm{u}-\overline{\bm{\alpha}}\|_{r,\partial X_+} \le \sqrt{C\varepsilon + \delta^2}, \quad \|\bm{u}-\overline{\bm{\alpha}}\|_{r,\partial X} \le \sqrt{C\varepsilon + \delta^2}.\label{eq:reg_assump_2_b}
	\end{equation}\label{eq:reg_assump_2}
\end{assumption}
\begin{remark}
	From \Cref{cor:est} and setting $\bm{\alpha}_0=\overline{\bm{\alpha}}$, we know that the above assumptions are true when $r=0$. Assuming $\sigma_{\mathrm{t}}(\bx)$, $\sigma_{\mathrm{s}}(\bx)$ and $f(\bx)$ are smooth enough 
	and taking any partial derivative $\partial_i$ ($i=1,2,3$) of \eqref{eq:rte_scale_compact} with respect to $x$, $y$ or $z$, respectively, we have
		\begin{alignat*}{3}
		\bOmega\cdot\nabla(\partial_i\bm{u}) + Q(\partial_i\bm{u}) &= \varepsilon\partial_i\bm{f}-(\partial_i Q)\bm{u}, &~~~&\text{ in }~ X \\
		\partial_i\bm{u} &= \bm{\beta}, &~~~&\text{ on }~ \begin{bmatrix}\partial X_-^1 & \partial X_-^2 &\cdots & \partial X_-^L\end{bmatrix}^{\tran}.
		\end{alignat*}
	where $\beta^l := \lim\limits_{\epsilon\to 0^+}\partial_i u^l(\bx+\epsilon\bomega_l)$ on $\partial X_-^l$. 
	With estimate \eqref{eq:priori_est_1_aniso}, we have
	\begin{align*}
		\frac{1}{\varepsilon}&\|\partial_i\bm{u}-\overline{\partial_i\bm{u}}\|^2 + \varepsilon\|\overline{\partial_i\bm{u}}\|^2 + \|\partial_i\bm{u}-\bm{\beta}_0\|^2_{\partial X} \\ 
		&\le c\left(\|\overline{\partial_i\bm{f}}\|^2 + \|\overline{\bm{u}}\|^2 + \|\bm{u}-\overline{\bm{u}}\|^2 +  \|\bm{\beta}_0\|^2_{1/2, \partial X}\right)\varepsilon + \|\bm{\beta}_1\|^2_{\partial X_-} + \varepsilon^3\|\partial_i\bm{f}-\overline{\partial_i\bm{f}}\|^2.
	\end{align*}
	If $\|\bm{\beta}_1\|^2_{\partial X_-}\approx \|\bm{\alpha}_1\|^2_{\partial X_-}$, then \Cref{cor:est} remains true for $\partial_i\bm{u}$, i.e., Assumptions \ref{assupt:reg_1} and \ref{assupt:reg_2} hold for $r=1$.
\end{remark}
\begin{assumption}\label{assupt:trace}
	$\overline{\bm{\alpha}}$ is the trace of a function in $\bm{V}_h$.
\end{assumption}
\begin{remark}
	\Cref{assupt:trace} is not strictly necessary. 
	If $\overline{\bm{\alpha}}$ is not the trace of a function in $\bm{V}_h$, we can construct an approximation $\overline{\bm{\alpha}}_h$ of $\overline{\bm{\alpha}}$. Let $\bm{u}'$ and $\bm{u}'_h$ be the solutions defined in \eqref{eq:rte_var_form} and \eqref{eq:rte_scale_DO_DG}, respectively, with the isotropic part of the boundary condition replaced by $\overline{\bm{\alpha}}_h$. Then $\|\bm{u}-\bm{u}'_h\| \le \|\bm{u}-\bm{u}'\| + \|\bm{u}'-\bm{u}'_h\|$. It can be shown that $\|\bm{u}-\bm{u}'\|$ is uniformly bounded by $\|\overline{\bm{\alpha}}-\overline{\bm{\alpha}}_h\|$. In fact, set $\bm{w}=\bm{u}-\bm{u}'$. Then $\bm{w}$ satisfies
	\begin{alignat*}{3}
	\bOmega\cdot\nabla\bm{w} + Q\bm{w} &= \bm{0}, &~~~&\text{ in }~ X\\
	\bm{w}&=\overline{\bm{\alpha}}-\overline{\bm{\alpha}}_h, &~~~&\text{ on }~ \begin{bmatrix}\partial X_-^1, & \partial X_-^2, &\cdots, & \partial X_-^L\end{bmatrix}^{\tran}.
	\end{alignat*}
	By employing \Cref{lem:pri_1}, we have
	\[ \|\bm{w}\|\le c(1+\varepsilon)\left\|\overline{\bm{\alpha}}-\overline{\bm{\alpha}}_h\right\|_{1/2,\partial X}. \] 
	Therefore the rest of the paper is unchanged but for perturbations involving $\|\overline{\bm{\alpha}}-\overline{\bm{\alpha}}_h\|_{1/2,\partial X}$. 
\end{remark}

Recall that, by \Cref{lem:cont}, $u^l$ is continuous on $\mathcal{E}^\mathrm{i}_h$ a.e.~if $\bomega_l\cdot\bm{n}\neq 0$, $l=1,\cdots,L$. Therefore, we have
	\begin{equation}\label{eq:cont}
	\bomega_l\cdot\bm{n}\,[\![u^l]\!]=0 \text{ on } \mathcal{E}^\mathrm{i}_h, \text{ a.e.,} \quad l=1,\cdots,L.
	\end{equation}

\subsection{General strategy}

The error analysis for $\bm{u}-\bm{u}_h$ is based on the following lemma, which is similar to the second Strang lemma (cf.~\cite{C2002}) in error analysis of nonconforming element methods. Readers can refer to \cite{HHE2010} for the proof of a scalar version.

\begin{lemma}\label{lem:stab_0}
	Let $\{\bm{W}_h\}_{h>0}:=\left\{[W_h]^L\right\}_{h>0}$ be a family of finite-dimensional product spaces equipped with norms $\{\|\cdot\|_h\}_{h>0}$. Let $\mathfrak{b}_h(\cdot,\cdot)$ be a uniformly coercive bi-linear form over $\bm{W}_h\times \bm{W}_h$, i.e., there exists a positive constant $\gamma$ independent of $h$ such that
	\begin{equation}\label{eq:stab_lem}
	\gamma\|\bm{w}_h\|^2_h\le \mathfrak{b}_h(\bm{w}_h,\bm{w}_h), \quad \forall \bm{w}_h\in \bm{W}_h.
	\end{equation}
	Let $\bm{Z}=[Z]^L$ be the product space of an (infinite dimensional) function space $Z$ and assume $\bm{v}$ is a vector of functions such that $\mathfrak{b}_h(\bm{v},\bm{w}_h)$ is well-defined and $|\mathfrak{b}_h(\bm{v},\bm{w}_h)|<C$ for all $\bm{w}_h\in \bm{W}_h$. Let $\bm{v}_h\in\bm{W}_h$ satisfy
	\begin{equation}\label{eq:rte_scale_lem_1_asspt2}
	\mathfrak{b}_h(\bm{v}-\bm{v}_h,\bm{w}_h)=0, \quad \forall \bm{w}_h\in \bm{W}_h.
	\end{equation}
	Then
	\begin{equation}\label{eq:error_inq1}
	\|\bm{v}-\bm{v}_h\|_h\le \inf_{\bm{w}_h\in \bm{W}_h}\left\{\|\bm{v}-\bm{w}_h\|_h + \frac{1}{\gamma}\sup_{\breve{\bm{w}}_h\in \bm{W}_h}\frac{\mathfrak{b}_h(\bm{v}-\bm{w}_h,\breve{\bm{w}}_h)}{\|\breve{\bm{w}}_h\|_h}\right\}.
	\end{equation}
\end{lemma}
We now construct an upper bound for  \eqref{eq:error_inq1}.
For any $K\in\mathcal{T}_h$, let $\mathcal{I}_h$ be an interpolation operator from $L^2(K)$ onto $P_k(K)$. 
We require that: (i) the restriction of $\mathcal{I}_h$ to any face $e\subset K$ is uniquely determined by the interpolation points on $e$; and (ii) two adjacent elements have the same interpolation points on their shared edge/surface.
Commonly employed nodal finite elements satisfy such requirement (see, e.g., \cite[Chapter 3]{BS2008} or \cite[Theorem 2.2.1]{C2002}). 
Then by a scaling argument and a trace theorem, we can easily obtain the following result (see, e.g., \cite{AH2009,BS2008,C2002}).
\begin{lemma}[Polynomial interpolation error estimate]
	Let $k=\deg(\mathcal{I}_h)$. Then for all $v\in H^r(K)$ with $r>0$ and $K\in \mathcal{T}_h$,
	\begin{equation}\label{eq:interp_est}
	\|(I-\mathcal{I}_h) v\|_{s,K}\lesssim h_K^{\min \{r,k+1\}-s}\|v\|_{r,K}, \; \|(I-\mathcal{I}_h) v\|_{0,\partial K}\lesssim h_K^{\min \{r,k+1\}-1/2}\|v\|_{r,K}.
	\end{equation}
\end{lemma}

Define the norm $|||\cdot|||_{h,\varepsilon}$ over $\bm{V}_{(h)}$ by
\begin{equation}\label{def:norm}
	|||\bm{v}_h|||^2_{h,\varepsilon}:= \varepsilon \mathfrak{a}_h(\bm{v}_h,\bm{v}_h)
	= \varepsilon \left( \| \bm{v}_h\|_Q^{2}
	+ \frac12 \| \bm{v}_h\|_{\partial X}^{2}
	+ \frac12 \big\|[\![\bm{v}_h]\!] \big\|_{\mathcal{E}^\mathrm{i}_h}^{2} \right).
\end{equation}
Here, the scale $\varepsilon$ is applied to compensate for the factor of $1/\varepsilon$ appeared in $\|\cdot\|_Q$. 

Below we establish the conditions of Lemma \ref{lem:stab_0} for the bilinear form $\mathfrak{a}_h(\cdot,\cdot)$, the norm $|||\cdot|||_{h,\varepsilon}$, and $\gamma= 1/\varepsilon$.
By \eqref{eq:dg_local} and the definitions \eqref{eq:def_a} and \eqref{eq:def_ell}, we have
\begin{equation}
    \mathfrak{a}_h(\bm{u},\bm{v}_h)=\ell(\bm{v}_h), \quad \forall \bm{v}_h\in \bm{V}_h.
\end{equation}
Subtract \eqref{eq:rte_scale_DO_DG} from the above equality to obtain the Galerkin orthogonality
\begin{equation}
    \mathfrak{a}_h(\bm{u}-\bm{u}_h,\bm{v}_h)=0, \quad  \forall \bm{v}_h\in \bm{V}_h.
\end{equation}
By \Cref{lem:a_h_stab} and \eqref{def:norm}, we have
\begin{equation}
	|||\bm{v}_h|||_{h,\varepsilon}^2 = \varepsilon \mathfrak{a}_h(\bm{v}_h,\bm{v}_h), \quad \forall\bm{v}_h\in \bm{V}_h.
\end{equation}
As a result,
\begin{equation}\label{eq:error_interp}
||| \bm{u}-\bm{u}_h|||_{h,\varepsilon} \le |||\bm{u}-\mathcal{I}_h \bm{u}|||_{h,\varepsilon} + \varepsilon\sup_{\bm{v}_h\in\bm{V}_h}\frac{\mathfrak{a}_h(\bm{u}-\mathcal{I}_h \bm{u},\bm{v}_h)}{|||\bm{v}_h|||_{h,\varepsilon}}.
\end{equation}

\subsection{Bounds for the terms in the right hand side of \texorpdfstring{\eqref{eq:error_interp}}{(45)}} 

\begin{lemma}\label{lem:interp_err}
	If Assumptions \ref{assupt:reg_1} and \ref{assupt:reg_2} hold, and further assume that, for $k=0$, $\bm{u}$ is Lipschitz continuous,
	then we have
	\begin{equation}\label{eq:interp_err}
	|||\bm{u}-\mathcal{I}_h \bm{u}|||_{h,\varepsilon} \lesssim \left\{\begin{array}{ll}  \sqrt{C'\varepsilon + 3(C\varepsilon^2 + \delta^2\varepsilon)}\, h^{\min\{r,1\}}, & \text{if } k=0,\\
	 \sqrt{3(C\varepsilon^2 + \delta^2\varepsilon)}\, h^{\min\{r,k+1\}}, & \text{if } k\ge 1.
	\end{array}\right. 
	\end{equation}
\end{lemma}
\begin{proof}
By the definition of $|||\cdot|||_{h,\varepsilon}$,
\begin{equation}
	|||\bm{u}-\mathcal{I}_h \bm{u}|||_{h,\varepsilon}^2
	= \varepsilon \| \bm{u}-\mathcal{I}_h \bm{u} \|_Q^2
	+ \frac{\varepsilon}{2} \| \bm{u}-\mathcal{I}_h \bm{u} \|_{\partial X}^2
	+ \frac{\varepsilon}{2} \big\| [\![\bm{u}-\mathcal{I}_h \bm{u}]\!] \big\|_{\mathcal{E}^\mathrm{i}_h}^2.
\end{equation}
We will estimate each term on the right-hand side of the above equality. 

For the first term, the assumption in \eqref{eq:reg_assump} and the interpolation estimate \eqref{eq:interp_est} imply that
\begin{align}\label{eq:Q_eq:interp_err}
\| \bm{u}-\mathcal{I}_h \bm{u} \|_Q^2
	&= \frac{1}{\varepsilon}\|\sigma_{\mathrm{t}}^{1/2}(I-\mathcal{I}_h)(\bm{u}-\overline{\bm{u}})\|^2 + \varepsilon \|\sigma_{\mathrm{a}}^{1/2}(I-\mathcal{I}_h)\overline{\bm{u}}\|^2\nonumber\\
	&\le \frac{1}{\varepsilon}h^{2\min\{r,k+1\}}(C\veps^2 + \delta^{2}\varepsilon) + \varepsilon h^{2\min\{r,k+1\}}(C+\frac{\delta^2}{\varepsilon}) \nonumber\\
	&\leq  2h^{2\min\{r,k+1\}} (C\veps + \delta^{2}).
\end{align}
For the second term, since $\mathcal{I}_h \overline{\alpha}=\overline{\alpha}$, and by \eqref{eq:reg_assump_2_b}, we have
\begin{align}
\varepsilon\| \bm{u}-\mathcal{I}_h \bm{u} \|_{\partial X}^2 &= \varepsilon\| (I-\mathcal{I}_h)(\bm{u}-\overline{\bm{\alpha}}) \|_{\partial X}^2 
\lesssim h^{2\min\{r,k+1\}} \varepsilon\| \bm{u}-\overline{\bm{\alpha}} \|_{r,\partial X}^2 \nonumber \\
&\lesssim \varepsilon(C\varepsilon + \delta^2)  h^{2\min\{r,k+1\}}.
\end{align}

For the third term, let us denote by $K_1$ and $K_2$ the two elements sharing a surface $e$, and for a scalar component $u^l$ of $\bm{u}$, write $u^l_1:=u^l|_{K_1}$ and $u^l_2:=u^l|_{K_2}$. When $k\ge 1$,  we notice that, for the nodal-based interpolation $\mathcal{I}_h$, the interpolation on each surface for an element is uniquely determined by the interpolation nodes. 
Therefore, we have 
\begin{equation}\label{eq:int_face_0}
	\big|[\![(I-\mathcal{I}_h)u^l]\!]\big|_e
	= |(I-\mathcal{I}_h)(u^l_1-u^l_2) \big|_e 
	   =|(I-\tilde{\mathcal{I}}_h)[\![u^l]\!]|_e, 
\end{equation}
where $\tilde{\mathcal{I}}_h$ is the interpolation operator on $e$. 
Therefore, by \eqref{eq:cont}, for $k\ge 1$,
\begin{equation}\label{eq:int_face_1}
	\int_e |\bomega_l\cdot\bm{n}|\,|[\![u^l-\mathcal{I}_h u^l]\!]|^2\ud\bx
	= \int_e |\bomega_l\cdot\bm{n}|\,|(I-\tilde{\mathcal{I}}_h)[\![u^l]\!]|^2\ud\bx\ 
	= 0.
\end{equation}
However, for $k=0$, $\mathcal{I}_h u^l$ is a piecewise constant, additional discontinuities with order of $O(1)$ are created along each $e\in \mathcal{E}^\mathrm{i}_h$. Owing to the assumption that $u^l$ is Lipschitz continuous, we can further have $|[\![(I-\mathcal{I}_h)u^l]\!]| \le \sqrt{C'}h$, where $C'$ is a constant. Therefore, for $k=0$, 
\begin{equation}
	\sum_{l=1}^L w_l \sum_{e\in \mathcal{E}^\mathrm{i}_h}\int_e |\bomega_l\cdot\bm{n}|\,|[\![u^l-\mathcal{I}_h u^l]\!]|^2\ud\bx
	= \sum_{e\in \mathcal{E}^\mathrm{i}_h}\sum_{l=1}^L w_l \int_e |\bomega_l\cdot\bm{n}|\,|[\![(I-\mathcal{I}_h) u^l]\!]|^2\ud\bx 
	\lesssim C' h^2,
\end{equation}
from which the third term can be estimated by
\begin{equation}
	\varepsilon\big\| [\![\bm{u}-\mathcal{I}_h \bm{u}]\!] \big\|_{\mathcal{E}^\mathrm{i}_h}^2 = \varepsilon\sum_{l=1}^L w_l \sum_{e\in \mathcal{E}^\mathrm{i}_h}\int_e |\bomega_l\cdot\bm{n}|\,\big|[\![u^l-\mathcal{I}_h u^l]\!]\big|^2\ud\bx \left\{\begin{array}{ll}
	\le \varepsilon C' h^2, & \text{if } k=0,\\
	=0, & \text{if } k\ge 1.
	\end{array}\right. 
\end{equation}
Therefore
\begin{equation}
	|||\bm{u}-\mathcal{I}_h \bm{u}|||_{h,\varepsilon}^2\lesssim \left\{\begin{array}{ll}  C'\varepsilon h^2 + 3(C\varepsilon^2 + \delta^2\varepsilon)\, h^{2\min\{r,1\}}, & \text{if } k=0,\\
	3(C\varepsilon^2 + \delta^2\varepsilon)\, h^{2\min\{r,k+1\}}, & \text{if } k\ge 1,
	\end{array}\right.
\end{equation}\
from which \eqref{eq:interp_err} follows.
\end{proof}

\begin{lemma} \label{lem:interp_bound}
	Under Assumptions~\ref{assupt:reg_1} and \ref{assupt:reg_2}, for $k\ge 1$, 
	\begin{multline}\label{eq:interp_bound}
		\varepsilon|\mathfrak{a}_h(\bm{u}-\mathcal{I}_h \bm{u},\bm{v}_h)|
		\lesssim
		h^{\min\{r,k+1\}}3 (C\varepsilon^2 + \delta^2\varepsilon)^{1/2}|||\bm{v}|||_{h,\varepsilon} \\
		+  h^{\min\{r,k+1\}-1} \left((C\varepsilon^4 + \delta^2\varepsilon^3)^{1/2} + 2\left(C\varepsilon^2 + \delta^2\varepsilon\right)^{1/2}\right)|||\bm{v}|||_{h,\varepsilon}.
	\end{multline}
\end{lemma}
\begin{proof}
Using the definition in \eqref{eq:def_a_2}, we write $\mathfrak{a}_h(\bm{u}-\mathcal{I}_h \bm{u},\bm{v}_h)= \mathrm{I} +\Pi + \mathrm{III} + \mathrm{IV}$,
where
\begin{align}\label{eq:a_decomp_1}
	\mathrm{I} &:=-\sum_{e\in \mathcal{E}^\mathrm{i}_h} \int_{e}\left(\bOmega \cdot \bm{n}_e [\![\bm{u} - \mathcal{I}_h \bm{u}]\!]\right)^{\tran} W \check{\bm{v}}_h \ud \bm{x}, \\
	\mathrm{II} &:= \sum_K  \int_{K}  \left(\bOmega \cdot \nabla  (\bm{u} - \mathcal{I}_h \bm{u})\right)^{\tran} W \bm{v}_h  \ud\bm{x} = \sum_K\left(\bOmega \cdot \nabla  (\bm{u} - \mathcal{I}_h \bm{u}),\, \bm{v}_h\right)_K, \\
	\mathrm{III} &:= ( \bm{u} - \mathcal{I}_h \bm{u} , \bm{v}_h)_{\partial X_-}, \quad 
	\mathrm{IV} :=  (\bm{u} - \mathcal{I}_h \bm{u} , \bm{v}_h)_Q.
\end{align}
To bound $\mathrm{I}$, we have, for $k\ge 1$, similar to \eqref{eq:int_face_1},
\begin{equation}\label{eq:error_I}
	\mathrm{I} = -\sum_{e\in \mathcal{E}^\mathrm{i}_h} \int_{e}\left(\bOmega \cdot \bm{n}_e \, (I-\tilde{\mathcal{I}}_h)[\![ \bm{u}]\!]\right)^{\tran} W \check{\bm{v}}_h \ud \bm{x} = 0.
\end{equation} 
The second term can be handled as follows:
\begin{align}
	|	\mathrm{II}|
		&\le \sum_K \big|\left(\bOmega\cdot\nabla (\bm{u}-\mathcal{I}_h \bm{u} - \overline{\bm{u}-\mathcal{I}_h \bm{u}} +\overline{\bm{u}-\mathcal{I}_h \bm{u}}), \,\bm{v}_{h}-\overline{\bm{v}_h}+\overline{\bm{v}_h}\right)_K\big|\nonumber\\
		&= \sum_K\big|\left(\bOmega\cdot\nabla (\bm{u}-\mathcal{I}_h \bm{u} - \overline{\bm{u}-\mathcal{I}_h \bm{u}}),\,\bm{v}_{h}-\overline{\bm{v}_h}\right)_K \nonumber \\
		&\phantom{= } \quad + \left(\bOmega\cdot\nabla (\bm{u}-\mathcal{I}_h \bm{u} - \overline{\bm{u}-\mathcal{I}_h \bm{u}}),\,\overline{\bm{v}_h}\right)_K\nonumber\\
		&\phantom{= } \quad + \left(\bOmega\cdot\nabla (\overline{\bm{u}-\mathcal{I}_h \bm{u}}),\,\bm{v}_{h}-\overline{\bm{v}_h}\right)_K + \underbrace{\left(\bOmega\cdot\nabla(\overline{\bm{u}-\mathcal{I}_h \bm{u}}),\,\overline{\bm{v}_h}\right)_K}_{\substack{=0 \text{ since }\overline{u^l-\mathcal{I}_h u^l}\\ \text{ and }\overline{v_h}\text{ are isotropic}}}\big|\nonumber\\
		&= \sum_K \big|\left(\bOmega\cdot\nabla (I-\mathcal{I}_h)(\bm{u}-\overline{\bm{u}}),\,\bm{v}_{h}-\overline{\bm{v}_h}\right)_K + \left(\bOmega\cdot\nabla (I-\mathcal{I}_h)(\bm{u}-\overline{\bm{u}}),\,\overline{\bm{v}_h}\right)_K\nonumber\\
		&\phantom{= } \quad + \left(\bOmega\cdot\nabla (I-\mathcal{I}_h)\overline{\bm{u}}, \,\bm{v}_{h}-\overline{\bm{v}_h}\right)_K\big| \nonumber\\
		&\le \sum_K \|\bOmega\cdot\nabla (I-\mathcal{I}_h)(\bm{u} - \overline{\bm{u}})\|_K\,\|\bm{v}_h-\overline{\bm{v}_h}\|_K + \sum_K \|\bOmega\cdot\nabla (I-\mathcal{I}_h)(\bm{u} - \overline{\bm{u}})\|_K\,\|\overline{\bm{v}_h}\|_K \nonumber\\
		&\phantom{= } \quad + \sum_K \|\bOmega\cdot\nabla(I-\mathcal{I}_h)(\overline{\bm{u}})\| \,\|\bm{v}_h-\overline{\bm{v}_h}\|_K\nonumber\\
		&\lesssim \frac{1}{h}h^{\min\{r,k+1\}}\left[C\varepsilon^2 + \delta^2\varepsilon\right]^{1/2}\|\bm{v}_h-\overline{\bm{v}_h}\|
		+\frac{1}{h}h^{\min\{r,k+1\}}\left[C\varepsilon^2 + \delta^2\varepsilon\right]^{1/2} \|\overline{\bm{v}_h}\|\nonumber\\
		&\phantom{= }\quad  +\frac{1}{h}h^{\min\{r,k+1\}}\left[C+\frac{\delta^2}{\varepsilon}\right]^{1/2}\|\bm{v}_h-\overline{\bm{v}_h}\| \nonumber\\
		&\lesssim \frac{1}{h}h^{\min\{r,k+1\}}\left[C\varepsilon^2 + \delta^2\varepsilon\right]^{1/2} |||\bm{v}|||_{h,\varepsilon}
		+\frac{1}{h}h^{\min\{r,k+1\}}\left[C+\frac{\delta^2}{\varepsilon}\right]^{1/2} |||\bm{v}|||_{h,\varepsilon} \nonumber\\
		&\phantom{= }\quad  +\frac{1}{h}h^{\min\{r,k+1\}}\left[C+\frac{\delta^2}{\varepsilon}\right]^{1/2} |||\bm{v}|||_{h,\varepsilon} \nonumber\\
		&\lesssim h^{\min\{r,k+1\}-1} \left((C\varepsilon^2 + \delta^2\varepsilon)^{1/2} + \left(C+\frac{\delta^2}{\varepsilon}\right)^{1/2} + \left(C+\frac{\delta^2}{\varepsilon}\right)^{1/2}\right) |||\bm{v}|||_{h,\varepsilon}\nonumber\\
		&\lesssim \frac{1}{\varepsilon} h^{\min\{r,k+1\}-1} \left((C\varepsilon^4 + \delta^2\varepsilon^3)^{1/2} + 2\left(C\varepsilon^2+ \varepsilon\delta^2\right)^{1/2} \right) |||\bm{v}|||_{h,\varepsilon}.
\end{align}
The third term is bounded using \Cref{assupt:trace}, which implies that $\mathcal{I}_h \overline{\alpha}(\bx)=\overline{\alpha}(\bx)$.  Hence
\begin{align}
	|\mathrm{III}| & =
	\big|\big((\bm{u}-\overline{\bm{\alpha}}) + \underbrace{(I-\mathcal{I}_h)\overline{\bm{\alpha}}}_{\substack{=0 \text{ based on} \\ \text{\Cref{assupt:trace}}}} + \mathcal{I}_h(\overline{\bm{\alpha}}-\bm{u}),\,\bm{v}_{h}\big)_{\partial X_-}\big|\nonumber\\
	&= \big|\big((I-\mathcal{I}_h) (\bm{u}-\overline{\bm{\alpha}}),\, \bm{v}_{h}\big)_{\partial X_-}\big|\nonumber\\
	&\le h^{\min\{r,k+1\}}\|\bm{u}-\overline{\bm{\alpha}}\|_{r,\partial X_-}\|\bm{v}_h\|_{\partial X_-}\nonumber\\
	&\le \frac{1}{\varepsilon} h^{\min\{r,k+1\}}\sqrt{C\varepsilon + \delta^2}\sqrt{\varepsilon}|||\bm{v}_h|||_{h,\varepsilon}.
\end{align}
The fourth term is bounded using \eqref{eq:Q_eq:interp_err}:
\begin{equation}\label{eq:error_IV}
	|\mathrm{IV}|
	\le \|\bm{u}-\mathcal{I}_h\bm{u}\|_Q \|\bm{v}_h\|_Q 
	\le \frac{1}{\varepsilon}h^{\min\{r,k+1\}}\sqrt{2(C\veps^2 + \delta^{2}\varepsilon)}|||\bm{v}|||_{h,\varepsilon}.
\end{equation}
By  combining \eqref{eq:error_I} -- \eqref{eq:error_IV}, \eqref{eq:interp_bound} follows.
\end{proof}

\begin{lemma} \label{lem:interp_bound_k=0}
	Under Assumptions~\ref{assupt:reg_1} and \ref{assupt:reg_2}, for $k=0$, 
	\begin{multline}\label{eq:interp_bound_k=0}
	\varepsilon|\mathfrak{a}_h(\bm{u}-\mathcal{I}_h \bm{u},\bm{v}_h)|
	\lesssim
	h^{\min\{r,1\}} 3 (C\varepsilon^2 + \delta^2\varepsilon)^{1/2}|||\bm{v}|||_{h,\varepsilon} \\
	+h^{\min\{r,1\}-1/2} \Big(\sqrt{C\varepsilon^3 + \delta^2\varepsilon^2} + \sqrt{C\varepsilon+\delta} \Big)|||\bm{v}|||_{h,\varepsilon} \\
	+  h^{\min\{r,1\}-1}\left((C\varepsilon^4 + \delta^2\varepsilon^3)^{1/2} + 2\left(C\varepsilon^2 + \delta^2\varepsilon\right)^{1/2}\right)|||\bm{v}|||_{h,\varepsilon}.
	\end{multline}
\end{lemma}
\begin{proof}
	Using the definition \eqref{eq:def_a} of $\mathfrak{a}_h$, we write
	$\mathfrak{a}_h(\bm{u}-\mathcal{I}_h \bm{u},\bm{v}_h)= \mathrm{I} +\Pi + \mathrm{III} + \mathrm{IV}$,
	where
	\begin{align}\label{eq:a_decomp_1_k=0}
	\mathrm{I} &:=\big((I-\mathcal{I}_h) \widehat{\bm{u}},\, [\bm{v}_{h}]\big)_{\mathcal{E}^\mathrm{i}_h}, \\
	\mathrm{II} &:= - \sum_K  \int_{K}  (\bm{u} - \mathcal{I}_h \bm{u})^{\tran} W (\bOmega \cdot \nabla  \bm{v}_h )  \ud\bm{x}, \\
	\mathrm{III} &:=  ( \bm{u} - \mathcal{I}_h \bm{u} , \bm{v}_h)_{\partial X_+}, \quad 
	\mathrm{IV} :=  (\bm{u} - \mathcal{I}_h \bm{u} , \bm{v}_h)_Q.
	\end{align}
	To bound $\mathrm{I}$, we have
	\begin{align}
	|\mathrm{I}|
	&= \big|\big((I-\mathcal{I}_h) \widehat{(\bm{u}-\overline{\bm{u}})},\, [\bm{v}_{h}]\big)_{\mathcal{E}^\mathrm{i}_h} + \big((I-\mathcal{I}_h) \widehat{\overline{\bm{u}}},\, [\bm{v}_{h}]\big)_{\mathcal{E}^\mathrm{i}_h}\big|\nonumber\\
	&\le \|(I-\mathcal{I}_h)(\bm{u}-\overline{\bm{u}})\|_{\mathcal{E}^\mathrm{i}_h} \|[\bm{v}_{h}]\|_{\mathcal{E}^\mathrm{i}_h} + \|(I-\mathcal{I}_h)\overline{\bm{u}}\|_{\mathcal{E}^\mathrm{i}_h} \|[\bm{v}_{h}]\|_{\mathcal{E}^\mathrm{i}_h} \nonumber\\
	&\le h^{\min\{r,1\}-1/2} \big(\|\bm{u}-\overline{\bm{u}}\|_{r,X} + \|\overline{\bm{u}}\|_{r,X}\big)\|[\bm{v}_{h}]\|_{\mathcal{E}^\mathrm{i}_h} \nonumber\\
	&\le \frac{1}{\varepsilon}h^{\min\{r,1\}-1/2} \Big(\sqrt{C\varepsilon^3 + \delta^2\varepsilon^2} + \sqrt{C\varepsilon+\delta} \Big)|||\bm{v}|||_{h,\varepsilon}.\label{eq:error_I_k=0}
	\end{align}
	The second term can be handled using an inverse inequality:
	\begin{align}
	|\mathrm{II}|
	&= \sum_K\Big|\left(\bm{u}-\mathcal{I}_h \bm{u} - \overline{\bm{u}-\mathcal{I}_h \bm{u}},\,\bOmega\cdot\nabla (\bm{v}_{h}-\overline{\bm{v}_h})\right)_K \nonumber \\
	&\phantom{= } \qquad + \left(\bm{u}-\mathcal{I}_h \bm{u} - \overline{\bm{u}-\mathcal{I}_h \bm{u}},\,\bOmega\cdot\nabla \overline{\bm{v}_h}\right)_K\nonumber\\
	&\phantom{= } \qquad + \left(\overline{\bm{u}-\mathcal{I}_h \bm{u}},\,\bOmega\cdot\nabla (\bm{v}_{h}-\overline{\bm{v}_h})\right)_K + \underbrace{\left(\overline{\bm{u}-\mathcal{I}_h \bm{u}},\,\bOmega\cdot\nabla\overline{\bm{v}_h}\right)_K}_{\substack{=0 \text{ since }\overline{u^l-\mathcal{I}_h u^l}\\ \text{ and }\overline{v_h}\text{ are isotropic}}}\Big|\nonumber\\
	&\le \sum_K \Big(\|(I-\mathcal{I}_h)(\bm{u} - \overline{\bm{u}})\|_K\,\|\bOmega\cdot\nabla (\bm{v}_h-\overline{\bm{v}_h})\|_K \nonumber\\
	&\phantom{= } \quad + \|(I-\mathcal{I}_h)(\bm{u} - \overline{\bm{u}})\|_K\,\|\bOmega\cdot\nabla \overline{\bm{v}_h}\|_K + \|(I-\mathcal{I}_h)(\overline{\bm{u}})\|_K\,\|\bOmega\cdot\nabla (\bm{v}_h-\overline{\bm{v}_h})\|_K \Big)\nonumber\\
	&\le \frac{1}{h}\sum_K \Big(\|(I-\mathcal{I}_h)(\bm{u} - \overline{\bm{u}})\|_K\,\|\bm{v}_h-\overline{\bm{v}_h}\|_K + \|(I-\mathcal{I}_h)(\bm{u} - \overline{\bm{u}})\|_K\,\|\overline{\bm{v}_h}\|_K \nonumber\\
	&\phantom{= } \quad + \|(I-\mathcal{I}_h)(\overline{\bm{u}})\|_K\,\|\bm{v}_h-\overline{\bm{v}_h}\|_K\Big) \nonumber\\
	&\lesssim \frac{1}{\varepsilon} h^{\min\{r,k+1\}-1} \left((C\varepsilon^4 + \delta^2\varepsilon^3)^{1/2} + 2\left(C\varepsilon^2+ \varepsilon\delta^2\right)^{1/2}\right) |||\bm{v}|||_{h,\varepsilon}.
	\end{align}
	For the third term, similar to the third term in \Cref{lem:interp_bound}, since $\mathcal{I}_h \overline{\alpha}(\bx)=\overline{\alpha}(\bx)$, we have
	\begin{equation}
	|\mathrm{III}| 
	= \big|\big((I-\mathcal{I}_h) (\bm{u}-\overline{\bm{\alpha}}),\, \bm{v}_{h}\big)_{\partial X_+}\big|
	\le \frac{1}{\varepsilon} h^{\min\{r,k+1\}}\sqrt{C\varepsilon^2 + \delta^2\varepsilon}|||\bm{v}_h|||_{h,\varepsilon}.
	\end{equation}	
	The fourth term is the same as the one in \Cref{lem:interp_bound}, which can be estimated by:
	\begin{equation}\label{eq:error_IV_k=0}
	|\mathrm{IV}| \le \frac{1}{\varepsilon}h^{\min\{r,k+1\}}\sqrt{2(C\veps^2 + \delta^{2}\varepsilon)}|||\bm{v}|||_{h,\varepsilon}.
	\end{equation}
	Hence, by combining \eqref{eq:error_I_k=0} -- \eqref{eq:error_IV_k=0}, \eqref{eq:interp_bound_k=0} follows.
\end{proof}

\subsection{Summary of results}
The final estimate for $||| \bm{u}-\bm{u}_h|||_{h,\varepsilon}$ follows from \eqref{eq:error_interp}, \Cref{lem:interp_err}, \Cref{lem:interp_bound}, and \Cref{lem:interp_bound_k=0}. 

\begin{theorem}\label{thm:main_theorem}
If Assumptions \ref{assupt:reg_1} and \ref{assupt:reg_2} hold, and further assume that for $k=0$, $\bm{u}$ is Lipschitz continuous,  then we have, for $k=0$,
\begin{multline} \label{eq:est_total_general_k=0}
	||| \bm{u}-\bm{u}_h|||_{h,\varepsilon} \lesssim \sqrt{C'}\varepsilon^{1/2}h + h^{\min\{r,1\}}\left(6(C\varepsilon^2 + \delta^2\varepsilon)^{1/2}\right)  \\
	+h^{\min\{r,1\}-1/2} \Big(\sqrt{C\varepsilon^3 + \delta^2\varepsilon^2} + \sqrt{C\varepsilon+\delta} \Big) \\
	+ h^{\min\{r,1\}-1} \left((C\varepsilon^4 + \delta^2\varepsilon^3)^{1/2} + 2\left(C\varepsilon^2+ \delta^2\varepsilon\right)^{1/2}\right);
\end{multline}
for  $k\ge 1$,
\begin{multline}\label{eq:est_total_general_k=1}
	||| \bm{u}-\bm{u}_h|||_{h,\varepsilon} \lesssim h^{\min\{r,k+1\}} 6 (C\varepsilon^2 + \delta^2\varepsilon)^{1/2} \\
	+ h^{\min\{r,k+1\}-1} \left((C\varepsilon^4 + \delta^2\varepsilon^3)^{1/2} + 2\left(C\varepsilon^2+ \delta^2\varepsilon\right)^{1/2}\right).
\end{multline}
\end{theorem}

The definition of $||| \bm{u}-\bm{u}_h|||_{h,\varepsilon}$ norm implies the following $L^2$ error bounds.
\begin{corollary}
Under the assumptions of \Cref{thm:main_theorem}, for $k=0$,
\begin{multline}\label{eq:err_anis_k=0}
	\left\| \bm{u}-\bm{u}_h \right\| \lesssim \sqrt{C'}(\varepsilon^{1/2}+\varepsilon^{-1/2})h + 6h^{\min\{r,1\}}\left((C\varepsilon^2 + \delta^2\varepsilon)^{1/2}+(C + \delta^2\varepsilon^{-1})^{1/2}\right)  \\
	+h^{\min\{r,1\}-1/2} \Big(\sqrt{C\varepsilon^3 + \delta^2\varepsilon^2} + 2\sqrt{C\varepsilon+\delta} + \sqrt{C\varepsilon^{-1}+\delta\varepsilon^{-2}} \Big) \\
	+ h^{\min\{r,1\}-1} \left((C\varepsilon^4 + \delta^2\varepsilon^3)^{1/2} + 3\left(C\varepsilon^2+ \delta^2\varepsilon\right)^{1/2} + 2\left(C+ \delta^2\varepsilon^{-1}\right)^{1/2}\right);
\end{multline}%
for $k\ge 1$,
\begin{multline}\label{eq:err_anis_k=1}
	\left\| \bm{u}-\bm{u}_h \right\| \lesssim 6 h^{\min\{r,k+1\}} \left((C\varepsilon^2 + \delta^2\varepsilon)^{1/2} + (C + \delta^2\varepsilon^{-1})^{1/2} \right) \\
	+ h^{\min\{r,k+1\}-1} \left((C\varepsilon^4 + \delta^2\varepsilon^3)^{1/2} + 3\left(C\varepsilon^2+ \delta^2\varepsilon\right)^{1/2} + 2\left(C+ \delta^2\varepsilon^{-1}\right)^{1/2}\right).
\end{multline}
If, in addition the boundary condition is isotropic, i.e. $\delta = 0$, then
$k=0$,
\begin{multline}\label{eq:err_iso_k=0}
	\left\| \bm{u}-\bm{u}_h \right\| \lesssim (\varepsilon^{1/2}+\varepsilon^{-1/2})h + h^{\min\{r,1\}}\left( \varepsilon + 1 \right) \\
	+ h^{\min\{r,1\}-1/2}\left(\varepsilon^{3/2} + \varepsilon^{1/2} + \varepsilon^{-1/2}\right) + h^{\min\{r,1\}-1}\left(\varepsilon^2 + \varepsilon + 1\right);
\end{multline}
for $k\ge 1$,
\begin{equation}\label{eq:err_iso_k=1}
	\left\| \bm{u}-\bm{u}_h \right\| \lesssim  h^{\min\{r,k+1\}}\left(\varepsilon + 1\right) + h^{\min\{r,k+1\}-1}\left(\varepsilon^2 + \varepsilon + 1\right).
\end{equation}
\end{corollary}

It is well know that $V_h^0$ (the function space that consists of piecewise constants) does not achieve the diffusion limit \cite{LMM1987}. 
This fact is reflected by several terms in \cref{eq:err_anis_k=0} that scales poorly with $\varepsilon$.  Even when $\delta = 0$, several of these terms remain. For $k \geq 1$ the only ``bad'' term in \eqref{eq:err_anis_k=1} is the $O(\delta \varepsilon^{-1/2})$ term that appears due to the anisotropy in the boundary condition.  However, even when the boundary condition is isotropic ($\delta=0$), the established convergence rate is not quite optimal. In the next section, we show that the rate can be made optimal for one-dimensional slab geometries and isotropic boundary conditions.  We also discuss an approach to control the boundary layer error when $\delta \ne 0$.

\section{Error analysis for the case of one-dimension slab geometry}\label{sec:1d_AP_error_analysis}
The error analysis above follows the framework developed in \Cref{sec:asym_error_analysis} and is applicable to fairly general settings. However, in one-dimensional slab geometries, better convergence results can be derived. Since it is already known that DG methods do not perform well when $k=0$, we focus here on the case $k\ge 1$.

In slab geometry \cite{LM1984}, the RTE takes the form
\begin{subequations}\label{eq:rte_scale_1d_all}
	\begin{align}
	\mu\frac{\partial u}{\partial x}+\frac{\sigma_{\mathrm{t}}(x)}{\varepsilon} u &=  \left(\frac{\sigma_{\mathrm{t}}(x)}{\varepsilon}-\varepsilon\sigma_{\mathrm{a}}(x)\right) \int^{1}_{-1} u(x,\hat{\mu})\ud\hat{\mu} + \varepsilon f,\quad x\in I:=(a,b),\\
	u(a,\mu)&=\alpha_{\mathrm{l}}(\mu) \text{ if } \mu>0,\quad u(b,\mu)=\alpha_{\mathrm{r}}(\mu) \text{ if } \mu<0.
	\end{align}
\end{subequations}
where $\mu\in[-1,1]$ is the $x$-coordinate of $\bomega$, $u=u(x,\mu)$, $f=f(x)$, and $\ud\hat{\mu}$ is the normalized measure on $(-1,1)$.  The discrete-ordinate equation can still be written in the form \eqref{eq:rte_scale_compact} so that all the notations and formulas in \Cref{sec:settings} can be kept.

\subsection{Error analysis}

Following \cite{CDG2008}, our analysis relies on the Radau projection $\mathcal{R}_h$.  Given a direction $\mu_l  \ne 0 $%
\footnote{In slab geometries, quadratures with the ordinate $\mu_l=0$ are rarely used in practice; even so, $\mathcal{R}_h$ can be defined as the usual $L^2$-orthogonal  projection in this case, since the advection terms vanish.} and an interval $I = (a,b)$, let $x^{\mathrm{out}}_I = (b+a)/2 + \operatorname{sgn}(\mu_l)(b-a)/2$ be the outflow point of $I$.  Then for $k\ge 1$, $\mathcal{R}_h$ is uniquely defined by the conditions
\begin{subequations}\label{eq:Radau_proj}
	\begin{align}
	(\mathcal{R}_h u-u,v)_I &=0, \quad \forall v\in P^{k-1}(I), \label{eq:Radau_proj_a}\\
	\mathcal{R}_h u(x^{\mathrm{out}}_{I}) &= u(x^{\mathrm{out}}_{I}).\label{eq:Radau_proj_b}
	\end{align}
\end{subequations}

\begin{lemma}[see, e.g., \cite{CDG2008}]
	Assume $u\in H^{r}(I)$. Then on each interval $I$,
	\begin{equation}\label{}
	\|u-\mathcal{R}_h u\|_{I}\le C h^{\min \{r,k+1\}}\|u\|_{r,I},
	\end{equation}
	where $C$ depends only on $k$.
\end{lemma}

We apply \Cref{lem:stab_0} with $\bm{w}_h=\mathcal{R}_h\bm{u}$ and use the $\|\cdot\|_Q$ norm. Then
\begin{equation}\label{eq:error_interp_2}
\| \bm{u}-\bm{u}_h\|_{Q} \le \|\bm{u}-\mathcal{R}_h \bm{u}\|_{Q} + \sup_{\bm{v}_h\in\bm{V}_h}\frac{\mathfrak{a}_h(\bm{u}-\mathcal{R}_h \bm{u},\bm{v}_h)}{\|\bm{v}_h\|_{Q}}.
\end{equation}
An estimate for the first term on the right-hand side of \eqref{eq:error_interp_2} is already obtained in \Cref{lem:interp_err}, i.e.,
\begin{equation}\label{eq:interp_err_1d}
	\|\bm{u}-\mathcal{R}_h \bm{u}\|_{Q} \le \sqrt{2(C\varepsilon + \delta^2)}\, h^{\min\{r,k+1\}}.
\end{equation}
For the second term, we have the following lemma.
\begin{lemma} \label{lem:interp_bound_2}
	Under \Cref{assupt:reg_1}, 
	\begin{equation}\label{eq:interp_bound_2}
	\mathfrak{a}_h(\bm{u}-\mathcal{R}_h \bm{u},\bm{v}_h)
	\lesssim \sqrt{C\varepsilon + \delta^2}\, h^{\min\{r,k+1\}}\|\bm{v}\|_{Q}.
	\end{equation}
\end{lemma}
\begin{proof}
	As in \Cref{lem:interp_bound}, we write $\mathfrak{a}_h(\bm{u}-\mathcal{R}_h \bm{u},\bm{v}_h)= \mathrm{I} +\Pi + \mathrm{III} + \mathrm{IV}$,
	where
	\begin{align}
	\mathrm{I} &:=-\sum_{e\in \mathcal{E}^\mathrm{i}_h} \int_{e}(\bOmega \cdot \bm{n}_e \widehat{(\bm{u} - \mathcal{R}_h \bm{u})})^{\tran} W [\![\bm{v_h}]\!] \ud \bm{x}, \\
	\mathrm{II} &:= - \sum_K  \int_{K}  (\bm{u} - \mathcal{R}_h \bm{u})^{\tran} W (\bOmega \cdot \nabla  \bm{v}_h )  \ud\bm{x}, \\
	\mathrm{III} &:=  ( \bm{u} - \mathcal{R}_h \bm{u} , \bm{v}_h)_{\partial X_+}, \quad 
	\mathrm{IV} :=  (\bm{u} - \mathcal{R}_h \bm{u} , \bm{v}_h)_Q.
	\end{align}
	However, now because of the property \eqref{eq:Radau_proj_b}, $\mathrm{I}=\mathrm{III}=0$,
	and from \eqref{eq:Radau_proj_a}, $\mathrm{II}=0$.
	Meanwhile, the argument from \Cref{lem:interp_bound} gives
	\begin{equation}
		\mathrm{IV}\le h^{\min\{r,k+1\}}\sqrt{2(C\veps + \delta^{2})}\|\bm{v}_h\|_{Q},
	\end{equation}
	from which \eqref{eq:interp_bound_2} follows.
\end{proof}

Together \eqref{eq:interp_err_1d} and \eqref{eq:interp_bound_2} imply that
\begin{equation}
	\| \bm{u}-\bm{u}_h\|_{Q} \lesssim \sqrt{C\varepsilon + \delta^2}\, h^{\min\{r,k+1\}},
\end{equation}
from which, we have the following theorem. 
\begin{theorem}
	Under \Cref{assupt:reg_1}, for $k \geq 1$,
	\begin{equation}\label{eq:err_est_l2_1d}
	\left\|\bm{u}-\bm{u}_h\right\| \lesssim \left(\sqrt{C\varepsilon^2 + \delta^2\varepsilon} + \sqrt{C + \delta^2/\varepsilon}\right)h^{\min\{r,k+1\}}.
	\end{equation}
	Furthermore, if $\delta=0$, then
	\begin{equation}
	\left\|\bm{u}-\bm{u}_h\right\| \lesssim (\varepsilon + 1)h^{\min\{r,k+1\}}.
	\end{equation}
\end{theorem}
Thus we obtain the optimal convergence order of $k+1$ for $k\ge 1$ when the boundary condition is isotropic.

\subsection{Anisotropic boundary conditions in one-dimension parallel slab geometry case}\label{subsec:aniso_bc}
When the boundary condition $\bm{\alpha}$ is anisotropic, we use the blended approach proposed in \cite{M2013}, which combines the given kinetic boundary condition with the isotropic boundary condition that results in the diffusion limit.  This condition is undertstood as the leading order approximation of the kinetic distribution after the boundary layer transition \cite{habetler1975uniform}. 

Let $\bm{v}$ solve
\begin{subequations}\label{eq:rte_scale_compact_v}
	\begin{align}
		\bomega\cdot\nabla\bm{v} + Q\bm{v} &= \varepsilon\bm{f},\quad \text{ in } I\\
		\bm{v}&=\lambda\bm{\alpha}(\bx,\bomega)+(1-\lambda)\bm{\alpha}_{\mathrm{b}}(\bx),\; \text{ on } \begin{bmatrix}\partial I_-^1 & \partial I_-^2 &\cdots & \partial I_-^L\end{bmatrix}^{\tran},
	\end{align}
\end{subequations}
where $\bm{\alpha}(\bx,\bomega)=\begin{bmatrix}\alpha^1 & \alpha^2 & \cdots & \alpha^L\end{bmatrix}^\tran$, $\lambda$ is a parameter which will be determined later, and $\bm{\alpha}_{\mathrm{b}}(\bx)$ is the boundary condition for the diffusion limit, whose $l$th component $[\bm{\alpha}_{\mathrm{b}}]_l$ is given by \cite{GJL1999}:
\begin{equation}\label{eq:boundary_corrector}
[\bm{\alpha}_{\mathrm{b}}]_l=\begin{cases}
\int_{0}^{\pi/2}W(\cos(\theta)) \alpha_l(\theta)\sin(\theta)\ud\theta, & \text{ if } \mu^l>0, x=a, \\
\int_{-\pi/2}^{0}W(\cos(\theta)) \alpha_r(\theta)\sin(\theta)\ud\theta, & \text{ if } \mu^l<0, x=b.
\end{cases}
\end{equation}
Here $W(\mu)=\frac{\sqrt{3}}{2}\mu H(\mu)$ is defined in terms of Chandrasekhar's H-function for isotropic scattering in a conservative medium \cite{MP1991}. 

The boundary condition for $\bm{v}$ is a convex combination of the original anisotropic boundary condition and an isotropic boundary correction.  Based on the decay of the boundary layer, $\lambda$ was set to $1-\exp(-\sigma_{\mathrm{t}}/\varepsilon)$ in   \cite{M2013} so that $\lambda\approx 1$ when $\varepsilon\ll 1$ and $\lambda\approx 0$ when $\sigma_{\mathrm{t}}\approx 0$. Here we instead choose $\lambda$ based on a balance between discretization and boundary layer errors.
To this end, we  further decompose the boundary condition as (see \Cref{lem:pri_1})
$\bm{\alpha}(\bx,\bomega)=\bm{\alpha}_0(\bx)+\bm{\alpha}_1(\bx,\bomega)$.
We choose $\bm{\alpha}_0(\bx)=\bm{\alpha}_{\mathrm{b}}(\bx)$. 
Then the boundary condition for $\bm{v}$ can be rewritten as
\begin{equation}
	 \bm{v}=\bm{\alpha}_{\mathrm{b}}(\bx) + \lambda\bm{\alpha}_1(\bx,\bomega),\quad \text{ on } \partial X_-.
\end{equation}
We consider $\bm{v}_h$ as the numerical approximation of $\bm{u}$ and analyze the error between $\bm{u}$ and $\bm{v}_h$, which, by triangle inequality, can be decomposed into
\begin{equation}
	 \|\bm{u}-\bm{v}_h\| \le \|\bm{u}-\bm{v}\|+\|\bm{v}-\bm{v}_h\|.
\end{equation}
To estimate $\|\bm{u}-\bm{v}\|$, set $\bm{e}:= \bm{u}-\bm{v}$. Then $\bm{e}$ satisfies
\begin{subequations}\label{eq:rte_scale_compact_e}
	\begin{align}
		\bomega\cdot\nabla\bm{e} + Q\bm{e} &= 0,\quad \text{ in } X\\
		\bm{e}&=(1-\lambda)(\bm{\alpha}(\bx,\bomega)-\bm{\alpha}_{\mathrm{b}}(\bx)),\quad \text{ on } \partial X_-.
	\end{align}
\end{subequations}
According to \cite[Theorem B.1]{GJL1999},
\begin{equation}
	 |e_l|\lesssim \frac{1-\lambda}{2-\mu_l}\delta_\infty\exp\left({-\frac{x}{2\varepsilon}}\right),
\end{equation}
where $\delta_\infty:=\max_{(\bx,\bomega)\in \partial X_-}|\bm{\alpha}(\bx,\bomega)-\bm{\alpha}_{\mathrm{b}}(\bx)|$.
Therefore,
\begin{equation}\label{eq:1d_part_I}
	 \|\bm{u}-\bm{v}\|=\|\bm{e}\| \lesssim (1-\lambda)\delta_\infty\sqrt{\varepsilon}.
\end{equation}
To estimate $\|\bm{v}-\bm{v}_h\|$, 
using \eqref{eq:err_est_l2_1d} and noting that $\|\bm{\alpha}_1(\bx,\bomega)\|_{\partial X_-}=\delta$, we deduce that
\begin{equation}\label{eq:1d_part_II}
	\left\|\bm{v}-\bm{v}_h\right\|
	\lesssim h^{\min\{r,k+1\}} \Big((C\varepsilon^2 + \lambda^2\delta^2\varepsilon)^{1/2}
	+ (C + \lambda^2\delta^2/\varepsilon)^{1/2}\Big).
\end{equation}
Combining \eqref{eq:1d_part_I} and \eqref{eq:1d_part_II} and noting that $\delta \lesssim \delta_\infty$ give
%
\begin{align}\label{eq:est_total_aniso}
	\|\bm{u}-\bm{v}_h\|
	 &\lesssim (1-\lambda)\delta_\infty\sqrt{\varepsilon} 
	 + h^{q} (C\varepsilon^2 + \lambda^2\delta_\infty^2\varepsilon)^{1/2}
	+h^{q} (C + \lambda^2\delta_\infty^2/\varepsilon)^{1/2}\nonumber\\
	&\lesssim (1-\lambda)^2\delta_\infty^2\varepsilon + h^{2k+2}\left(C + \lambda^2\delta_\infty^2/\varepsilon\right),
\end{align}
where $q=\min\{r,k+1\}$.
This bound has a minimum with respect to $\lambda$ when
\begin{equation}\label{eq:lambda_min}
	\lambda^* = \frac{\varepsilon^2}{\varepsilon^2 + \beta},\quad \beta:= h^{2q}.
\end{equation}
If $h/\varepsilon\to 0$, then $\lambda\to 1$, which is expected since in this case the boundary layer will be fully resolved. Substituting \eqref{eq:lambda_min} into \eqref{eq:est_total_aniso} and keeping the dominant terms in $\varepsilon$, we have
\begin{align*}
	\|\bm{u}-\bm{v}_h\|
	&\lesssim \left(1-\frac{\varepsilon^2}{\varepsilon^2 + \beta}\right)\delta_\infty\sqrt{\varepsilon} + \sqrt{\beta}\left(C+\left(\frac{\varepsilon^2}{\varepsilon^2 + \beta}\right)^2 \frac{\delta_\infty^2}{\varepsilon}\right)^{1/2}\\
	&\lesssim \left(\frac{\beta}{\varepsilon^2 + \beta}\right)\delta_\infty\sqrt{\varepsilon} + \sqrt{\beta}\left(\sqrt{C} + \frac{\varepsilon^{3/2}\delta_\infty}{\varepsilon^2 + \beta} \right)\\
	&= \sqrt{\beta}\sqrt{C}+\left(\frac{\beta}{\varepsilon^2 + \beta}\right)\delta_\infty\sqrt{\varepsilon} + \delta_\infty\sqrt{\varepsilon}\frac{\varepsilon\sqrt{\beta}}{\varepsilon^2 + \beta},
\end{align*}
from which we have a uniform convergence with respect to $\varepsilon$. In fact, $\left(\frac{\beta}{\varepsilon^2 + \beta}\right)\sqrt{\varepsilon}$ has a maximum $\frac{1}{2}\beta^{1/4}$ at $\beta=\varepsilon^2$, and $\sqrt{\varepsilon}\frac{\varepsilon\sqrt{\beta}}{\varepsilon^2 + \beta}$ has a maximum $\frac{1}{4}\beta^{1/4}$ when $\beta=\frac{1}{3}\varepsilon^2$. Hence, in the worst-case scenario, when $\beta=O(\varepsilon^2)$, we have
\begin{equation}\label{eq:1d_est_wrst}
	 \|\bm{u}-\bm{v}_h\|\lesssim O(\beta^{\frac{1}{2}})+\delta_\infty O(\beta^{\frac{1}{4}})\approx O(h^{k}) + \delta_\infty O(h^{\frac{k}{2}}).
\end{equation}

\section{Numerical results}\label{sec:num_results}
In this section, we present some numerical results.
We focus only on the one-dimension slab geometry problems \eqref{eq:rte_scale_1d_all}, since we can also investigate the analysis of the boundary layer effect.

In all numerical examples, the computation domain is $X=(-1.0, 1.0)$. Let $\mathcal{T}_0=\mathcal{T}_{h_0}$ be an initial triangulation of $X$ with $8$ equal elements each of which has a mesh size $h_0 = 2/8$.
Then we recursively generate nested interval cells $\mathcal{T}_j=\mathcal{T}_{h_j}$, $j=1,2,3,\cdots$,  by dividing each interval cell in the previous mesh $\mathcal{T}_{j-1}$ into two equal sub-intervals. The boundary conditions are specified by $\alpha_{\mathrm{l}}(\mu)$ for $\mu>0$ and $x=-1$ and by $\alpha_{\mathrm{r}}(\mu)$ for $\mu<0$ and $x=1$, respectively. A GMRES solver with a diffusion synthetic accelerator (DSA)  \cite{CCGH2017,Larsen2010} is employed to solve the discrete-ordinate equations.

\begin{example}\label{ex:1}
We take $\sigma_{\mathrm{t}}=2$ and $\sigma_{\mathrm{a}}=1$. The source function $f= f_{\mathrm{bump}}$ where $f_{\mathrm{bump}}$ is a mollifier bump function with support radius $r=0.125$, i.e.,
\begin{equation*}
	f_{\mathrm{bump}} = \begin{cases}
	\exp \left(\frac{1}{(x/r)^2-1}\right), &\text{ if } |x|<r, \\
	0, &\text{ if } |x|\ge r,
	\end{cases}
\end{equation*}
and $\alpha_{\mathrm{l}}(\mu)=0.1$ and $\alpha_{\mathrm{r}}(\mu)=0$.  We use the solution on mesh $T_{8}$ ($16384$ cells) as the reference (true) solution and compute the errors for solutions obtained on the coarse meshes.
The results using linear elements are reported in \Cref{tab:AP_1D_1}, which confirm our analysis.

\begin{table}[!htbp]
\begin{center}
\caption{Error and Convergence rate for $k=1$ and $k=2$ for Example~\ref{ex:1}}\label{tab:AP_1D_1}
\begin{tabular*}{0.98\textwidth}{@{\extracolsep{\fill} } c | c c | c c | c c }
\toprule[1pt]
    & \multicolumn{6}{ c }{$k=1$}\\
\midrule[0.5pt]
    & \multicolumn{2}{ c | }{$\varepsilon=10^0$} & \multicolumn{2}{ c | }{$\varepsilon=10^{-3}$} & \multicolumn{2}{ c }{$\varepsilon=10^{-5}$} \\
\midrule[0.5pt]
$h$ &$\|\bm{u}^*_{h_1}-\bm{u}\|$ & rate &$\|\bm{u}^*_{h_2}-\bm{u}\|$ & rate &$\|\bm{u}^*_{h_2}-\bm{u}\|$ & rate\\
\midrule[1pt]
$2/2^3$             &1.93e-01 &/    &4.00e-02 &/    &4.02e-02 &/    \\
\midrule[0.5pt]
$2/2^4$             &6.26e-02 &1.62 &9.52e-03 &2.07 &9.65e-03 &2.06 \\
\midrule[0.5pt]
$2/2^5$             &2.08e-02 &1.59 &3.28e-03 &1.54 &3.35e-03 &1.52 \\
\midrule[0.5pt]
$2/2^6$             &6.12e-03 &1.76 &7.44e-04 &2.14 &7.82e-04 &2.10 \\
\midrule[0.5pt]
$2/2^7$             &1.56e-03 &1.97 &1.60e-04 &2.22 &1.76e-04 &2.15 \\
\midrule[0.5pt]
$2/2^8$             &4.02e-04 &1.96 &3.66e-05 &2.12 &4.39e-05 &2.01 \\
\midrule[1.0pt]
 $h$   & \multicolumn{6}{ c }{$k=2$}\\
\midrule[0.5pt]
$2/2^3$             &8.52e-02 &/    &1.40e-02 &/    &1.41e-02 &/ \\
\midrule[0.5pt]
$2/2^4$             &3.70e-02 &1.20 &3.44e-03 &2.03 &3.46e-03 &2.02 \\
\midrule[0.5pt]
$2/2^5$             &9.82e-03 &1.91 &5.52e-04 &2.64 &5.55e-04 &2.64 \\
\midrule[0.5pt]
$2/2^6$             &9.75e-04 &3.33 &2.50e-05 &4.46 &2.62e-05 &4.41 \\
\midrule[0.5pt]
$2/2^7$             &1.15e-04 &3.09 &3.15e-06 &2.99 &3.42e-06 &2.94 \\
\midrule[0.5pt]
$2/2^8$             &1.65e-05 &2.80 &5.18e-07 &2.60 &3.82e-07 &3.16 \\
\bottomrule[1pt]
\end{tabular*}
\end{center}
\end{table}

\end{example}

\begin{example}\label{ex:2}
The following example is mainly to illustrate the performance of the blended boundary condition when the true boundary condition is anisotropic. We set $\alpha_{\mathrm{l}}(\mu)=0.1+\mu/100$, $\alpha_{\mathrm{r}}(\mu)=0$, $k=1$, and $h_0=2/32$.  The other settings remain the  same as in Example~\ref{ex:1}. From the definition $\alpha_{\mathrm{b}}(x)$ in \eqref{eq:boundary_corrector}, 
\begin{equation*}
    \alpha_{\mathrm{b}}(x)|_{x=-1}=\int_{0}^{1}W(\mu) \alpha_l(\mu)\ud\mu\\
     \approx 0.10710446089598763,
\end{equation*}
where we use the fact that $\int_{0}^{1}W(\mu)\ud\mu=1$ and $\int_{0}^{1}\mu W(\mu) \ud\mu \approx 0.710446089598763$ (see \cite{MP1991}).

\begin{table}[!htbp]
	\begin{center}
		\caption{Error, $\lambda_{\mathrm{min}}$, $\lambda^*$, and Convergence rate for $\varepsilon=10^{-1}$ for Example~\ref{ex:2}}\label{tab:AP_1D_2_1}
		\begin{tabular*}{0.85\textwidth}{@{\extracolsep{\fill} } c | c c c | c c c}
		\toprule[1pt]
		$h$ &$\|\bm{u}^{\mathrm{min}}_{h}-\bm{u}\|$ & $\lambda_{\mathrm{min}}$ & rate & $\|\bm{u}^*_{h}-\bm{u}\|$ & $\lambda^*$ & rate \\
		\midrule[1pt]
		$2/2^6$             &5.93016e-04 & 1.0	&/  &5.93016e-04 & 0.99999 &/	\\
		\midrule[0.5pt]
		$2/2^7$             &1.52914e-04 & 1.0	& 1.96 &1.52914e-04	& 0.99999 &1.96  \\
		\midrule[0.5pt]
		$2/2^8$             &4.63630e-05 & 1.0	& 1.72 &4.63630e-05	& 0.9999996 &1.72	\\
		\midrule[0.5pt]
		$2/2^9$             &1.62808e-05 & 1.0	& 1.51 &1.62808e-05	& 0.99999998 &1.51	\\
		\bottomrule[1pt]
		\end{tabular*}
	\end{center}
\end{table}

\begin{table}[!htbp]
	\begin{center}
		\caption{Error, $\lambda_{\mathrm{min}}$, $\lambda^*$, and Convergence rate for $\varepsilon=10^{-2}$ for Example~\ref{ex:2}}\label{tab:AP_1D_2_2}
		\begin{tabular*}{0.85\textwidth}{@{\extracolsep{\fill} } c | c c c | c c c }
		\toprule[1pt]
		$h$ &$\|\bm{u}^{\mathrm{min}}_{h}-\bm{u}\|$ & $\lambda_{\mathrm{min}}$ & rate & $\|\bm{u}^*_{h}-\bm{u}\|$ & $\lambda^*$ & rate \\
		\midrule[1pt]
		$2/2^6$             &5.76542e-04 & 0.76 &/ &5.76866e-04	 & 0.991 &/	  \\
		\midrule[0.5pt]
		$2/2^7$             &1.26336e-04 & 0.9 &2.19 &1.26657e-04 & 0.9994 &2.19	 \\
		\midrule[0.5pt]
		$2/2^8$             &4.66157e-05 & 0.98 &1.44 &4.66666e-05 & 0.99996 &1.44	 \\
		\midrule[0.5pt]
		$2/2^9$             &2.39469e-05 & 0.999 &0.96 &2.39472e-05 & 0.999998 &0.96	 \\
		\bottomrule[1pt]
		\end{tabular*}
	\end{center}
\end{table}

\begin{table}[!htbp]
	\begin{center}
		\caption{Error, $\lambda_{\mathrm{min}}$, $\lambda^*$, and convergence rate for $\varepsilon=10^{-3}$ for Example~\ref{ex:2}}\label{tab:AP_1D_2_3}
		\begin{tabular*}{0.85\textwidth}{@{\extracolsep{\fill} } c | c c c | c c c }
		\toprule[1pt]
		$h$ &$\|\bm{u}^{\mathrm{min}}_{h}-\bm{u}\|$ & $\lambda_{\mathrm{min}}$ & rate & $\|\bm{u}^*_{h}-\bm{u}\|$ & $\lambda^*$ & rate \\
		\midrule[1pt]
		$2/2^6$             &7.45808e-04 & 0.36 &/ &7.45870e-04	& 0.51186 &/	 \\
		\midrule[0.5pt]
		$2/2^7$             &1.62312e-04 & 0.24 &2.20 &1.65729e-04 & 0.943748 &2.17	 \\
		\midrule[0.5pt]
		$2/2^8$             &4.62701e-05 & 0.32 &1.81 &5.23629e-05 & 0.996289 &1.66	 \\
		\midrule[0.5pt]
		$2/2^9$             &2.69756e-05 & 0.53 &0.78 &3.05380e-05 & 0.999767 &0.78	 \\
		\bottomrule[1pt]
		\end{tabular*}
	\end{center}
\end{table}

For a fixed $\varepsilon$, we employ the solution of \eqref{eq:rte_scale_compact_v} with $\lambda=1$ on mesh $\mathcal{T}_{8}$ as the reference (true) solution $\bm{u}$ (since the mesh is sufficiently fine to resolve the solution in the boundary layer) and compute the errors for solutions $\bm{u}_{h}$ obtained on the coarse meshes with different values of $\lambda$. 
In fact, for a fix value of $\varepsilon$ and $h$, the error is approximately a quadratic function of $\lambda$ with a minimum at $\lambda_{\mathrm{min}}$ in $(0,1)$. 
We show the error for $\bm{u}^{\mathrm{min}}_{h}$ which minimizes $\|\bm{u}_{h}-\bm{u}\|$ the value of  $\lambda_{\mathrm{min}}$ that achieves it by sweeping different values of $\lambda$. Sufficiently many values of $\lambda$ are employed to guarantee that the errors with the chosen value $\lambda_{\mathrm{min}}$ are close enough to the true minimum errors for any $\lambda \in [0,1]$. 
For comparison, the errors between $\bm{u}$ and $\bm{u}^{*}_{h}$ which corresponds to $\lambda=\lambda^*$ defined in \eqref{eq:lambda_min} are also provided. 
The results are reported in 
\cref{tab:AP_1D_2_1,tab:AP_1D_2_2,tab:AP_1D_2_3}. 
It is noticeable that, since many approximations and simplifications are applied in the derivation of \eqref{eq:lambda_min}, $\lambda^*$ may not be close to $\lambda_{\mathrm{min}}$. However, uniform convergence can still be observed from the numerical results, with convergence rates are consistent with the theoretical estimate \eqref{eq:1d_est_wrst}.  (The estimated rate in this case is $h^{1/2}$.) This fact is important, since in practice one can choose $\lambda^*$, but not $\lambda_{\mathrm{min}}$.

\end{example}

\section{Conclusion}\label{sec:conclusion}
In this paper, we analyze the convergence of a discontinuous Galerkin scheme for the scaled discrete-ordinate radiative transfer equation with isotropic scattering kernel.  For sufficiently rich approximation spaces we prove uniform convergence rates with respect to $\varepsilon$ when the boundary is isotropic.  However,  this rate is not quite optimal.  In slab geometries with isotropic boundary conditions, we can obtain optimal and uniform convergence rate.  For anisotropic boundary conditions, we propose to solve an auxiliary problem and analyze the error between the numerical solutions of the auxiliary problem and the original. We show that by properly choosing the parameter in the auxiliary problem,  uniform convergence can  be achieved. Some numerical results are presented to demonstrate how these errors behave in practice.

\section*{Acknowledgments}
The authors would like to thank Michael Crockatt for providing his code for solving one-dimensional parallel slab problems.

\ifx\SIAM\TRUE
\bibliographystyle{siamplain}
\bibliography{uni_conv_DG_RTE}
\else
\bibliographystyle{amsplain}
\bibliography{uni_conv_DG_RTE}

\providecommand{\bysame}{\leavevmode\hbox to3em{\hrulefill}\thinspace}
\providecommand{\MR}{\relax\ifhmode\unskip\space\fi MR }
\providecommand{\MRhref}[2]{%
  \href{http://www.ams.org/mathscinet-getitem?mr=#1}{#2}
}
\providecommand{\href}[2]{#2}
\begin{thebibliography}{10}

\bibitem{Adams2001}
Marvin~L. Adams, \emph{Discontinuous finite element transport solutions in
  thick diffusive problems}, Nuclear Science and Engineering \textbf{137}
  (2001), no.~3, 298--333.

\bibitem{agoshkov1998boundary}
V.~Agoshkov, \emph{Boundary value problems for transport equations}, Modeling
  and Simulation in Science, Engineering and Technology, Birkh{\"a}user Boston,
  1998.

\bibitem{AH2009}
Kendall Atkinson and Weimin Han, \emph{Theoretical numerical analysis: A
  functional analysis framework}, 3rd ed., Texts in Applied Mathematics,
  vol.~39, Springer-Verlag New York, 2009.

\bibitem{bensoussan1979boundary}
Alain Bensoussan, Jacques~L Lions, and George~C Papanicolaou, \emph{{Boundary
  layers and homogenization of transport processes}}, Publications of the
  Research Institute for Mathematical Sciences \textbf{15} (1979), no.~1,
  53--157.

\bibitem{BS2008}
Susanne Brenner and Ridgway Scott, \emph{The mathematical theory of finite
  element methods}, 3rd ed., Texts in Applied Mathematics, vol.~15,
  Springer-Verlag New York, 2008.

\bibitem{CZ1967}
K.~M. Case and P.~F. Zweifel, \emph{Linear transport theory}, Addison-Wesley,
  Reading, MA, 1967.

\bibitem{C2002}
Phillipe~G. Ciarlet, \emph{The finite element method for elliptic problems},
  Classics in Applied Mathematics, Society for Industrial and Applied
  Mathematics, 2002.

\bibitem{CY2014}
James~A. {Coakley Jr.} and Ping Yang, \emph{Atmospheric radiation: A primer
  with illustrative solutions}, Wiley-VCH Verlag, Germany, 2014.

\bibitem{CDG2008}
Bernardo Cockburn, Bo~Dong, and Johnny Guzm\'{a}n, \emph{A superconvergent
  {LDG}-hybridizable {G}alerkin method for second-order elliptic problems},
  Mathematics of Computation \textbf{77} (2008), 1887--1916.

\bibitem{CCGH2017}
Michael~M. Crockatt, Andrew~J. Christlieb, C.~Kristopher Garrett, and Cory~D.
  Hauck, \emph{An arbitrary-order, fully implicit, hybrid kinetic solver for
  linear radiative transport using integral deferred correction}, Journal of
  Computational Physics \textbf{346} (2017), 212--241.

\bibitem{DM1978}
J.~J. Duderstadt and W.~R. Martin, \emph{Transport theory}, John Wiley, New
  York, 1978.

\bibitem{EG2004}
Alexandre Ern and Jean-Luc Guermond, \emph{Theory and practice of finite
  elements}, Applied Mathematical Sciences, vol. 159, Springer-Verlag, New
  York, 2004.

\bibitem{GJL1999}
François Golse, Shi Jin, and C.~David Levermore, \emph{The convergence of
  numerical transfer schemes in diffusive regimes {I}: Discrete-ordinate
  method}, SIAM Journal on Numerical Analysis \textbf{36} (1999), no.~5,
  1333--1369.

\bibitem{GK2010}
Jean-Luc Guermond and Guido Kanschat, \emph{Asymptotic analysis of upwind
  discontinuous {G}alerkin approximation of the radiative transport equation in
  the diffusive limit}, SIAM Journal on Numerical Analysis \textbf{48} (2010),
  no.~1, 53--78.

\bibitem{habetler1975uniform}
GJ~Habetler and BJ~Matkowsky, \emph{Uniform asymptotic expansions in transport
  theory with small mean free paths, and the diffusion approximation}, Journal
  of Mathematical Physics \textbf{16} (1975), no.~4, 846--854.

\bibitem{HHE2010}
W.~Han, J.~Huang, and J.~Eichholz, \emph{Discrete-ordinate discontinuous
  {G}alerkin methods for solving the radiative transfer equation}, SIAM Journal
  on Scientific Computing \textbf{32} (2010), no.~2, 477--497.

\bibitem{jin2010asymptotic}
Shi Jin, \emph{Asymptotic preserving ({AP}) schemes for multiscale kinetic and
  hyperbolic equations: a review}, Rivista di Matematica della Universit\`{a}
  di Parma \textbf{3} (2012), no.~2, 177--216.

\bibitem{LaSaint1974Finite}
P.~La~Saint and P.~A. Raviart, \emph{{On a Finite Element Method for Solving
  the Neutron Transport Equation}}, pp.~89--123, Elsevier, 1974.

\bibitem{LarsenKeller}
Edward~W. Larsen and Joseph~B. Keller, \emph{Asymptotic solution of neutron
  transport problems for small mean free paths}, Journal of Mathematical
  Physics \textbf{15} (1974), 75.

\bibitem{LM1989}
Edward~W Larsen and J.E. Morel, \emph{Asymptotic solutions of numerical
  transport problems in optically thick, diffusive regimes {II}}, Journal of
  Computational Physics \textbf{83} (1989), no.~1, 212--236.

\bibitem{LMM1987}
Edward~W Larsen, J.E Morel, and Warren F~Miller Jr., \emph{Asymptotic solutions
  of numerical transport problems in optically thick, diffusive regimes},
  Journal of Computational Physics \textbf{69} (1987), no.~2, 283--324.

\bibitem{Larsen2010}
Edward~W. Larsen and Jim~E. Morel, \emph{Advances in discrete-ordinates
  methodology}, pp.~1--84, Springer Netherlands, Dordrecht, 2010.

\bibitem{LM1984}
E.~E. Lewis and W.~F. Miller, \emph{Computational methods of neutron
  transport}, John Wiley \& Sons, New York, 1984.

\bibitem{LM2010}
Jian-Guo Liu and Luc Mieussens, \emph{Analysis of an asymptotic preserving
  scheme for linear kinetic equations in the diffusion limit}, SIAM Journal on
  Numerical Analysis \textbf{48} (2010), no.~4, 1474--1491.

\bibitem{lowrie2002methods}
RB~Lowrie and JE~Morel, \emph{Methods for hyperbolic systems with stiff
  relaxation}, International Journal for Numerical Methods in Fluids
  \textbf{40} (2002), no.~3-4, 413--423.

\bibitem{MP1991}
F.~Malvagi and G.~C. Pomraning, \emph{Initial and boundary conditions for
  diffusive linear transport problems}, Journal of Mathematical Physics
  \textbf{32} (1991), no.~3, 805--820.

\bibitem{M2013}
Luc Mieussens, \emph{On the asymptotic preserving property of the unified gas
  kinetic scheme for the diffusion limit of linear kinetic models}, Journal of
  Computational Physics \textbf{253} (2013), 138--156.

\bibitem{Modest}
Michael~F. Modest, \emph{Radiative heat transfer}, 3rd ed., Academic Press,
  2013.

\bibitem{Pe2004}
Annamaneni Peraiah, \emph{An introduction to radiative transfer: Methods and
  applications in astrophysics}, Cambridge University Press, 2001.

\bibitem{reed1973triangular}
William~H Reed and TR~Hill, \emph{Triangular mesh methods for the neutron
  transport equation}, Tech. report, Los Alamos Scientific Lab., N. Mex.(USA),
  1973.

\bibitem{TS1999}
G.~E. Thomas and K.~Stamnes, \emph{Radiative transfer in the atmosphere and
  ocean}, Cambridge University Press, 1999.

\bibitem{ZTB}
Wilford Zdunkowski, Thomas Trautmann, and Andreas Bott, \emph{Radiation in the
  atmosphere: A course in theoretical meteorology}, Cambridge University Press,
  2007.

\end{thebibliography}
\fi

\end{document}